\documentclass[11pt,bezier]{article}
 \usepackage{amsmath,amssymb,amsfonts,euscript,graphicx,amsthm}
                        \usepackage[all]{xy}
                        \usepackage{multicol}
                        \usepackage{enumerate}
                          \usepackage{tikz}

  \evensidemargin =   -  3 cm \topmargin = -0.5 cm
  \newtheorem{The}{Theorem}[section]
  \newtheorem{Pro}[The]{Proposition}
  \newtheorem{Lem}[The]{Lemma}
  \newtheorem{Cor}[The]{Corollary}
  \newtheorem{Def}[The]{Definition}
  \newtheorem{Defs}[The]{Definitions}
  \newtheorem{Rem}[The]{Remark}
  
  \newtheorem{Examp}[The]{Example}

    \let\oldproofname=\proofname
   \renewcommand{\proofname}{\textit{\rm\bf\oldproofname}}
\title{\LARGE\bf  Left  Co-K{\"o}the Rings and  Their Characterizations\thanks{The research of the second  author was in part supported by a grant from IPM  (No.1401130213). This research is partially carried out in the IPM-Isfahan Branch.}~
\thanks {2020{ \it Mathematics Subject Classification}. Primary 16D70, 16G60, 16D90; Secondary 16D10,  16P20. }}
\author{{\bf {\bf Shadi Asgari}$^{{\rm b}}$, Mahmood Behboodi$^{{\rm a,b,}}$\thanks {Corresponding author.}\hspace{1mm} and {\bf Somayeh Khedrizadeh}$^{{\rm a}}$}  \\
{\small{ $^{{\rm a}}$Department of Mathematical Sciences, Isfahan University of Technology}}\vspace{-1mm}\\
{\small{ P.O.Box :  84156-83111,   Isfahan,   Iran}}\\
{\small{ $^{{\rm b}}$School of Mathematics, Institute for Research in Fundamental Sciences
(IPM)}}\vspace{-1mm}\\ {\small{ P.O.Box : 19395-5746, Tehran, Iran}}\vspace{-1mm}\\
   {\small{sh{\_}asgari@ipm.ir}}\vspace{-1mm}\\
   {\small{mbehbood@iut.ac.ir}}\vspace{-1mm}\\
   {\small{s.khedrizadeh@math.iut.ac.ir}}}
  \date{}
  
\begin{document}
  \maketitle
 \vspace*{-0.5cm}
 \begin{abstract}
 \small{ 
\noindent K{\"o}the's classical problem posed by G. K{\"o}the  in 1935 asks to describe the rings $R$  such that every left  $R$-module is a direct sum of cyclic modules (these rings are known as left K{\"o}the rings). K{\"o}the, Cohen and Kaplansky solved this problem for all commutative rings (that are Artinian principal ideal rings).  During the years 1962 to 1965,  Kawada solved K{\"o}the's problem for basic fnite-dimensional algebras.  But, so far,  K{\"o}the's problem was open in the non-commutative setting. Recently, in the paper [{\it Several characterizations of left K{\"o}the rings}, submitted], we classified left K{\"o}the rings into three classes one contained in the other: {\it  left K{\"o}the rings},  {\it  strongly left K{\"o}the rings} and   {\it  very strongly left K{\"o}the rings}, and then, we solved K{\"o}the's problem by giving several characterizations of these rings in terms of describing the indecomposable modules.  In this paper, we will introduce the Morita duals of these  notions as  {\it  left co-K{\"o}the rings},  {\it  strongly left co-K{\"o}the rings} and   {\it  very strongly left co-K{\"o}the rings},   and then, we give several  structural  characterizations for each of them.}      \end{abstract}
%111111111111111111111111111111111111
%Introduction%Introduction%Introduction
%111111111111111111111111111111111111
 \section{\bf Introduction}    
 In 1935,  K{\"o}the  \cite{Kothe}  showed that over an Artinian principal ideal ring (i.e.,  a ring which is both a left and right Artinian,  and is both a left and a right principal ideal ring)   each module is a direct sum of cyclic modules, and  he posed the question to classify the rings  with this property. A ring for which each 
left (resp., right) module is a direct sum of cyclic modules, is now called a {\it left}  (resp., {\it right})  {\it K{\"o}the  ring}.   In 1941, Nakayama \cite{Nakayama1, Nakayama2} introduced the notion of generalized uniserial rings as a generalization of Artinian principal ideal  rings, and proved
that generalized uniserial rings are  K{\"o}the rings (a ring $R$  is called a {\it generalized uniserial ring}, if $R$ has a unit element  and if every left ideal $Re$ as well as every right ideal $eR$  generated by a primitive 
idempotent element $e$ possesses only one composition series). However, as is shown by Nakayama, the rings of this type are not general enough for solving K{\"o}the's problem (see  \cite[Page 289]{Nakayama2}).
  In 1951, Cohen and Kaplansky \cite{Cohen-Kaplansky} proved that all commutative K{\"o}the rings are Artinian principal ideal rings. Thus,  by combining the above results one obtains:
  \begin{The}
  {\rm (K{\"o}the,   \cite{Kothe})}. An Artinian principal ideal ring is a K{\"o}the ring.
   \end{The} 
   
   \begin{The}
  {\rm (K{\"o}the-Cohen-Kaplansky, \cite{Cohen-Kaplansky, Kothe}}).   A commutative ring $R$ is a K{\"o}the ring if and only if $R$ is an Artinian principal ideal ring.
   \end{The} 

 However,  a left Artinian principal left ideal ring $R$ need not be a left  K{\"o}the ring, even if $R$ is a local ring (see \cite[page 212, Remark (2)]{Faith}). Faith \cite{Faith} characterized all commutative rings,  whose proper factor rings are K{\"o}the rings. During the years 1962 to 1965,  Kawada \cite{Kawada1, Kawada2, Kawada3} solved K{\"o}the's problem for basic finite-dimensional algebras.   Behboodi et al. \cite{Behboodi1}  showed that  if $R$ is a left K{\"o}the ring in which all idempotents are central,  then $R$ is an Artinian principal right ideal ring.

  Throughout this paper, all rings have an identity element and all modules are unital. 
Let $M$ be an $R$-module.  The {\it socle} of $M$, denoted by $soc(M)$, is  
the sum of all simple  submodules of $M$. If there is no simple submodules in $M$ we put $soc(M) = 0$.  Dual to the socle the {\it radical}  of  $M$, denoted by $Rad(M)$, is the intersection of all maximal submodules of $M$. If $M$ has no maximal submodules we set $Rad(M) = M$.  If $M$ has a proper submodule which contains all other proper submodules, then $M$ is called a {\it local module} (i.e., $Rad(M)$ is a maximal submodule of $M$). A ring is a {\it local ring} if and only if $_RR$ (or $R_R$) is a local module. 
  For any  module $M$, the {\it top} of $M$, denoted by $top(M)$, is the factor module $M/Rad(M)$.   
 A module is called {\it square-free} if it does not contain a direct sum of two nonzero isomorphic submodules.  An Artinian module $M$ is 
square-free  if and only if $soc(M)$ is a  square-free module.  We will see that the square-free modules and decomposition  of (finitely generated)  modules into square-free modules are closely related to K{\"o}the's  problem. 

Most recently, in \cite{Behboodi3}, we provide several characterizations for a left K{\"o}the ring. In fact, in \cite{Behboodi3},  by using the following definitions, we classified left K{\"o}the rings into three classes one contained in the other: and then we  have solved  K{\"o}the's classical problem by giving several characterizations of these rings in terms of describing the indecomposable modules. To see the several structural characteristics of each of the following  rings  we refer the reader to  \cite[Theorems 2.6, 3.4, 3.5, 4.1 and 4.7 ]{Behboodi3}.
\begin{Defs}
{\rm We say that a ring $R$  is a:\\
{\it  - left}  (resp., {\it right})  {\it K{\"o}the ring}    if every left  (resp.,  right)  $R$-module is a  direct sum of   cyclic modules (i.e., $R$ is of finite representation type and every    indecomposable left $R$-module has a cyclic top by  \cite[Theorem 2.6]{Behboodi3}).\\
{\it - strongly left}  (resp., {\it right})  {\it K{\"o}the ring}    if every  nonzero left  (resp.,  right)  $R$-module is a  direct sum of  modules  with nonzero  square-free  cyclic top. \\
{\it - very  strongly  left}   (resp., {\it right}) {\it  K{\"o}the ring}  if every  nonzero left (resp.,  right)   $R$-module is a  direct sum of   modules with  simple top. \\
{\it  - K{\"o}the ring}  if $R$ is both  a  left and  a  right  K{\"o}the ring.\\
{\it - strongly  K{\"o}the ring}  if $R$ is both  a  strongly left and  a  strongly right  K{\"o}the ring.\\
{\it - very strongly  K{\"o}the ring}  if $R$ is both  a   very  strongly left and   right K{\"o}the  ring.}
\end{Defs}
 By   \cite[Theorem 2.6]{Behboodi3} a ring $R$ is a left  K{\"o}the  ring if and only if 
every  nonzero left $R$-module is a direct sum of  modules with nonzero socle and  cyclic top, if and only if, 
  $R$ is  left Artinian  and  every   left $R$-module  is a direct sum of modules with cyclic top. Also, in  \cite[Theorem 3.5]{Behboodi3},  
 we	have shown that a quasi-duo ring $R$ is a  left K{\"o}the ring if and only if every left $R$-module is a direct sum of modules with square-free top if and only if every left $S$-module is a direct sum of modules with square-free socle, where $S$ is a Morita dual ring of $R$ (see Remark \ref{left Morita} for the notion of Morita dual of a ring). These facts motivated us to define left co-K{\"o}the ring, strongly left co-K{\"o}the ring and very strongly  left co-K{\"o}the ring by replacing the top with the socle in Definitions 1.4 as follows:   

 \begin{Def}
{\rm We say that a ring $R$  is a:\\
{\it - left}  (resp., {\it right})  {\it co-K{\"o}the ring}    if every nonzero left   (resp.,  right)  $R$-module is a  direct sum of    modules with nonzero cyclic socle.\\
 {\it - strongly left}  (resp., {\it right})  {\it co-K{\"o}the ring}    if every nonzero left (resp.,  right)   $R$-module is a  direct sum of  modules  with nonzero  square-free  socle. \\
 {\it - very  strongly  left}   (resp., {\it right}) {\it  co-K{\"o}the ring}  if every nonzero left  (resp.,  right)  $R$-module is a  direct sum of   modules with  simple socle. \\
 {\it  - co-K{\"o}the ring}  if $R$ is both  a  left and  a  right  co-K{\"o}the ring.\\
 {\it - strongly  co-K{\"o}the ring}  if $R$ is both  a  strongly left and  a  strongly right  co-K{\"o}the ring.\\
 {\it - very strongly  co-K{\"o}the ring}  if $R$ is both  a   very  strongly left and right co-K{\"o}the  ring.}
\end{Def}
The main goal of this paper is to present several properties and  structural characterizations for each of the above concepts. Summing them all up leads to solving the Morita dual of K{\"o}the's problem.
 
This paper is organized as follows. In Section $2$, we give some preliminaries and propositions that are needed in the following sections. In Section $3$, we provide a new characterization  of left pure semisimple rings by using the socle of finitely generated modules (see Theorem \ref{left pure semisimple}). Also,   we give several characterizations of left co-K{\"o}the rings (see Theorem \ref{left co-Kothe}). In Section $4$, we present various structural characterizations of strongly left co-K{\"o}the ring, and we show that such rings coincide with rings on which every left module is a direct sum of square-free modules (see Theorem \ref{Theorem}). Also, we show that for a quasi-duo ring  $R$ the concepts of “strongly left co-K{\"o}the” and “left co-K{\"o}the” are the same (see Theorem \ref{left quasi duo left co-kothe}) which concludes in Corollary \ref{S is strongly right Kothe}  as a generalization of Ringel's Theorem \cite[Theorem 1.6]{Ringel}. In Section $5$,  we give various structural characterizations of  very strongly left co-K{\"o}the rings. Among other characterizations, we show that a ring $R$ is a very strongly left co-K{\"o}the ring if and only if every left $R$-module is a direct sum of extending modules, if and only if, every left $R$-module is a direct sum of uniform modules. Moreover, we show that any very strongly left co-K{\"o}the ring $R$ is an Artinian left serial ring (see Theorem \ref{very strongly left co-Kothe}). Also, we give some result about  K{\"o}the and co-K{\"o}the rings. In Section 6, some relevant examples and counterexamples are included to illustrate our results.
 %22222222222222222222222222222222222222
 %Preliminaries  and related background
 %22222222222222222222222222222222222222 
 \section{\bf Some preliminaries  and related background  }
 
 The module $M$ is called {\it distributive} if the lattice of all its submodules is distributive; i.e., $A \cap (B + C) = (A \cap B) + (A \cap C)$ for all submodules $A$, $B$, and $C$ of $M$. Also, $M$ is called semi-distributive whenever $M$ is a direct sum of distributive modules. For a detailed account of such modules,  see \cite{tuganbaev}. Obviously, uniform modules are square-free, and by a result of Stephenson \cite[Corollary 1(i)$'$ of Proposition 1.1]{Stephenson}, distributive modules are also square-free.
  By \cite[Proposition 1.2]{Fuller1},  a semisimple module  $M$ is square-free if and only if it is distributive if and only if it is zero or a direct sum of non-isomorphic simple modules. Thus,  we have the following. 
  
 \begin{Lem}\label{square-free}
Let $M$ be a left $R$-module. If $M$ is  Artinian or $R$ is left Artinian,  then the following statements are equivalent:\\
\indent {{\rm (1)}} $M$ is a square-free module.\\
\indent {{\rm (2)}} $soc(M)$ is a  square-free module.\\
\indent {{\rm (3)}} $soc(M)$ is a   distributive module.\\
\indent {{\rm (4)}} $soc(M)$ is a direct sum of non-isomorphic simple modules.\\
\indent {{\rm (5)}} Composition factors of any finitely generated submodule of  $soc(M)$ are pairwise \indent\indent non-isomorphic.
\end{Lem}
\begin{proof} 
(1) $\Rightarrow$ (2),  (4) $\Leftrightarrow$ (5) are clear, and (2) $\Leftrightarrow$ (3) $\Leftrightarrow$ (4) is~by~\cite[Proposition 1.2]{Fuller1}.\\
(2) $\Rightarrow$ (1).  Assume that $M$ is a square-free module.  If $0 \neq A \oplus B \subseteq M$ and $A \cong B$,  then we conclude that there are simple submodules $S_{1} \subseteq A$ and $S_{2} \subseteq B$ with $S_{1} \cong S_{2}$. It follows that    $S_1 \oplus S_2 \subseteq soc(M)$, a contradiction. 
  \end{proof}

Recall that  a ring $R$ is {\it semi-perfect}  if  $R/J$  is left semisimple and idempotents in $R/J$  can be
lifted to $R$.  So,  for example, local rings are semi-perfect. From \cite[(15.16), (15.19) and (27.1)]{Anderson-Fuller}
 it follows that a left (or right) Artinian ring is semi-perfect. It is worthy of note that in a semi-perfect ring,  the radical is the unique largest
ideal containing no nonzero idempotents (see  \cite[(15.12)]{Anderson-Fuller}).

Recall that if $R$  is a semi-perfect ring,  

A semi-perfect ring $R$ can be written as a direct sum of  indecomposable cyclic left $R$-modules,
written as $R= {(Re_1)}^{(t_1)}  \oplus\cdots\oplus {(Re_n)}^{(t_n)}$, where  $(Re_i)^{(t_i)}$ denotes the direct sum of $t_i$ copies of $Re_i$. 
 If $A= Re_1  \oplus\cdots\oplus Re_n$   then
the category of right $R$-modules and right $A$-modules are Morita equivalent. The
semi-perfect ring $R$ is said to be {\it basic}  if $t_1 =  \cdots = t_n = 1$, i.e., there are no isomorphic modules in any decomposition of the ring $R$ into a direct sum of indecomposable right $R$-modules. In fact, a semi-perfect ring $R$  is  basic if the quotient ring $R/J$ is a direct sum of division rings.

Also, a ring $R$ is left (right) {\it perfect} in case each of its left (right) modules has a projective cover. It follows from \cite[Proposition 27.6]{Anderson-Fuller} that left perfect rings and right
perfect rings are both semi perfect. However, left  perfect
rings need not be right perfect  (see \cite[Exercise (28.2)]{Anderson-Fuller}). 
The pioneering work on perfect rings was carried out by H. Bass in 1960 and most of the main characterizations of these rings are contained in his celebrated paper \cite{Bass}.
 
Recall that an idempotent $e\in R$  is {\it primitive} in case it is nonzero and
cannot be written as a sum $e = e^\prime + e^{\prime\prime}$ of nonzero orthogonal idempotents.
A left (right) ideal of $R$ is primitive in case it is of the form $Re$ ($eR$) for some
primitive idempotent $e \in R$. The endomorphism ring of $Re$ is isomorphic to $eRe$.

We recall that a set $\{e_1, \cdots , e_m\}$ of idempotents of a semiperfect ring $R$ is called {\it basic} in case they are pairwise orthogonal, $Re_i \ncong Re_j$ for each $i \neq j$ and for each finitely
generated indecomposable projective left $R$-module $P$, there exists $i$ such that $P \cong Re_i$.
Clearly, the cardinal of any two basic sets of idempotents of a semiperfect ring $R$ are equal. An idempotent $e$ of a semi-perfect ring $R$ is called a {\it basic idempotent} of $R$ in case $e$ is the sum $e = e_1 + \cdots + e_m$ of a basic set $e_1,\cdots, e_m$ of primitive idempotents of $R$. 
        
 Let $M$,  $N$ be $R$-modules.  A monomorphism   $ N \rightarrow M$   is called an  {\it embedding}  of  $N$  to  $M$,   and  we denote  it by $N\hookrightarrow M$).
 
We need the following proposition, which will use it in future sections
\begin{Pro}\label{Projective cover}
Let $R$ be a semi-perfect ring and $M$ be a  finitely generated left $R$-module with the projective cover $P(M)$. Then\\
 {\rm (a)} $R\cong \oplus_{i=1}^n(Re_i)^{(t_i)}$,  where $n,~t_i\in\Bbb{N}$ and $\{e_1,\cdots, e_n\}$ is a basic  set of idempotents     of $R,$ \indent  $R/J \cong \oplus_{i=1}^n (Re_i/Je_i)^{(t_i)}$, 
 $top(M)\cong \oplus_{i=1}^n (Re_{i}/Je_i)^{(s_i)}$,     for  some  
 $s_1,\cdots,s_n\in\Bbb{N}\cup \{0\}$,   \indent and 
  $P(M) \cong \oplus_{i=1}^n(Re_i)^{(s_i)}$. Consequently, \\
  {\rm (b)}   $soc(M) \cong (Re_1/Je_1)^{(u_1)}\oplus\cdots\oplus (Re_n/Je_n)^{(u_n)},$   for some
 $u_i\in\Bbb{N}\cup \{0\}$ where    $i= \indent 1,\cdots, n$.\\
  \indent In particular,  $soc(M)$  is cyclic (resp., square-free)  if and  only if $u_i\leq t_i$  ${(resp., u_i\in\{0,1\})}$  \indent for all $i=1,\cdots, n.$\\
 {\rm (c)}   If $soc(M)$ is  square-free, then $soc(M)$  is   cyclic. 
 \end{Pro}
\begin{proof}  The fact (a) is by \cite[Corollary 15.18, Theorems 27.6,  27.13   and Proposition 27.10]{Anderson-Fuller}.

 (b). The proof of the first statement  is obtained by \cite[Proposition 27.10]{Anderson-Fuller}, since $soc(M)$ is the  sum of simple modules.\\
  ($\Rightarrow$). It is well known that, $J(R) soc(M) =0$.  If $soc(M)$ is cyclic, then $soc(M)$ isomorphic to a summand of $R/J(R)$ and hence it can be embedded in $R/J(R)$. So,  $u_i\leq t_i$ for all $i=1,\cdots, n.$\\
  ($\Leftarrow$). If $u_i\leq t_i$ for all $i=1,\cdots, n$, then $soc(M)$  is isomorphic to a direct summand of $R/J$ and hence $soc(M)$ is a cyclic  module.
   
  (c). It is clear by Part (b).
  \end{proof}

In the sequel,  we denote by  $R$-Mod (resp., $R$-mod)  the category of  (resp., finitely generated)  left $R$-modules, and  by  Mod-$R$ (resp., mod-$R$)  the category of  (resp., finitely generated) right $R$-modules.
For an $R$-module $M$,  $E(M)$  denotes the {\it injective hull}  of $M$.
 
In the continuation of this section, we will explain some contents of Wisbauer's book \cite{Wisbauer} regarding the functors rings  of the
finitely generated modules of  $R$-Mod, which we will use extensively in the next sections.  
  
Let $\cal{U}$ be a non-empty set (class) of objects in a
category $\cal{C}$. An object $A$ in $\cal{C}$ is said to be {\it generated by}  $\cal{U}$ or {\it $\cal{U}$-generated}  if,
for every pair of distinct morphisms $f,~g : A \rightarrow B$  in $\cal{C}$, there is a morphism
$h : U \rightarrow A$ with $U\in \cal{U}$  and $hf \neq  hg$. In this case,  $\cal{U}$ is called a {\it set (class) of
generators for}  $A$.  Let $M$ be an $R$-module.  We  recall that   an $R$-module $N$ is
{\it subgenerated by $M$}, or that $M$ is a {\it subgenerator}  for $N$, if $N$ is isomorphic to
a submodule of an $M$-generated module. We denote by $\sigma[M]$ the full subcategory of $R$-Mod whose objects are
all $R$-modules subgenerated by $M$ (which is  the ``smallest"  subcategory of $R$-Mod which contains $M$ and is
a {\it Grothendieck category}).

Let $\{V_\alpha\}_A$  be a family of finitely generated $R$-modules and $V = \bigoplus_AV_\alpha$.
For any $N\in R$-Mod we define:
$$\widehat{H}om(V, N)=\{f\in Hom(V, N) ~|~ f(V_\alpha) = 0~{\rm for~ almost~ all~} \alpha\in A\}.$$
For $N=V$ , we write $\widehat{H}om(V, V )= \widehat{E}nd(V)$. Note that these constructions
do not depend on the decomposition of $V$ (see \cite[Chap. 10, Sec. 51]{Wisbauer} for more details).

For a left $R$-module $M$, let $\{U_\alpha\}_A$ be a representing set of the finitely
generated modules in $\sigma[M]$. We define $U=\bigoplus_AU_\alpha$, $T = \widehat{E}nd(U)$ and call $T$ the {\it functor ring of the
finitely generated modules of $\sigma[M]$}. Then by  \cite[Chap. 10, Sec. 52]{Wisbauer}, for any  module $M$,   $U$ is a generator in $\sigma[M]$ and  by  \cite[Page 507, Part (10) ]{Wisbauer}, 
the functor $\widehat{H}om(U, -)$  is an equivalence between the subcategories of the direct
summands of direct sums of finitely generated modules in $\sigma[M]$ and the
projective modules in $T$-Mod. 

\begin{Rem}\label{progenerator}
{\rm Let $M$ be a left $R$-module. If   $P$ is a finitely generated  projective generator in $\sigma[M]$ (i.e., $P$ is a progenerator in $\sigma[M]$), then by  \cite[Proposition 46.4]{Wisbauer}, 
 $Hom(P, -) : \sigma[M] \rightarrow End(P )$-Mod is an equivalence,  and so by \cite[Chap. 10, Sec. 52]{Wisbauer},  we have $End(U)\cong End(Hom(P, U))$ and
$$T = \widehat{E}nd (U)\cong \widehat{E}nd(Hom(P, U)) \cong \widehat{E}nd (\bigoplus Hom(P, U_\alpha)).$$
Since $\{Hom_R(P, U_\alpha)\}_A$  is a representing set of the finitely generated
$End(P)$-modules, in this case,  $\sigma[M]$ and $End(P)$-Mod have isomorphic
functor rings.}
\end{Rem}

Thus,  we conclude that  if $M=R$,  and $\{U_\alpha\}_A$ is  a representing set of the finitely generated modules in $R$-Mod,  then always $U=\bigoplus_AU_\alpha$ is a generator in $R$-Mod.
 Moreover,  if  $R$ is a  left Artinian ring,  then   every finitely generated left $R$-module  has finite length and hence is a direct sum of indecomposable modules.
So,  we have $U = \bigoplus_A U_\alpha=\bigoplus_\Lambda V_\lambda$,  where  $V_\lambda$ is a finitely generated indecomposable module. Since
$T$ is independent of the decomposition chosen for $U$, $T = \widehat{E}nd(\bigoplus_A U_\alpha) =\widehat{E}nd(\bigoplus_\Lambda V_\lambda)$ (see also  [58, Page 533, the proof  of (b)--(f) ]).
    So,  we have the following:

\begin{Pro}\label{left Artinian}
Let  $R$ be a left Artinian ring, $\{V_\alpha\}_A$ be  a complete set of representatives set of all finitely generated (indecomposable) left $R$-modules, $V=\bigoplus_AV_\alpha$  and $T=\widehat{E}nd(V)$.~Then\\
\indent  {\rm (a)}  $V$  is a generator in $R$-Mod.\\
\indent  {\rm (b)}  $\widehat{H}om (V, -)$ preserves essential extensions.\\
 \indent  {\rm (c)}  The functor ring  $T$ is semi-perfect.\\
  \indent  {\rm (d)}  $V_T$ is finitely generated projective in Mod-$T$ and Mod-$T = \sigma[V_T]$ (i.e., $V_T$ is a \indent\indent  progenerator in $\sigma[V_T]$).\\
 \indent  {\rm (e)}  $End(V_T )\cong End(V_{T^{\prime}})\cong R$, where   $T^{\prime}=End(V)$.\\
 \indent  {\rm (f)}  If $M$ is a simple left $R$-module,  then $\widehat{H}om(V, M)$ has a simple essential socle. \\
  \indent  {\rm (g)}   A finitely generated (indecomposable) left $R$-module $M$  has a square-free  (essential)   \indent\indent socle if and only if  $\widehat{H}om(V, M)$     has a square-free (essential)  socle.\\
  \indent  {\rm (h)}   A  left $R$-module $M$  has a simple (essential)  socle if and only if  $\widehat{H}om(V, M)$  has  \indent\indent  a simple  (essential) socle.\\
  \indent  {\rm (i)}  If $R$ is a left pure semisimple  ring,   then 
      $\widehat{H}om(V, -)$  is an equivalence between \indent\indent  the full categories of left  $R$-modules  and the
projective modules in $T$-Mod.   
    \end{Pro}  
 \begin{proof}  (a) is by \cite[Chap. 10, Sec. 52]{Wisbauer}, (b)  and (c) are  optioned by \cite[Proposition 51.7) ]{Wisbauer} and (d), (e)   and (f) are optioned  by  \cite[Proposition 51.8) ]{Wisbauer}
  
 (g). Since $R$ is  left Artinian,   for  every  finitely generated (indecomposable) left $R$-module $M$,  $soc(M)=S_1\oplus\cdots\oplus S_n\leq_e M$, where  $S_i^,s$ are simple. By (f)   each $\widehat{H}om_R(V,S_i)$   has  a simple  essential socle and by (b), 
    $$\bigoplus_{i=1}^n \widehat{H}om_R(V,S_i)= \widehat{H}om_R(V,soc(M))\leq_e \widehat{H}om_R(V,M)\leq_e \bigoplus_{i=1}^n \widehat{H}om_R(V,E(S_i)).$$  
  It follows that $soc(\widehat{H}om_R(V,M))= soc(\bigoplus_{1}^n \widehat{H}om_R(V,S_i))$ and is clear that   $S_i\ncong S_j$ if and only if $\widehat{H}om_R(V,S_i)\ncong   \widehat{H}om_R(V,S_j)$, if and only if, 
$soc(\widehat{H}om_R(V,S_i))\ncong   soc(\widehat{H}om_R(V,S_j))$.  Thus $M$   has a square-free  (essential)  socle if and only if  $\widehat{H}om(V, M)$     has a square-free (essential) socle.

  (h) is by the proof of (g) for $n=1$.
 
(i) By \cite[Theorem 4.4]{chase}, any left pure semisimple ring $R$ is left Artinian. So every left module has an indecomposable decomposition. Then, the proof is complete by \cite[Page 507, Part (10) ]{Wisbauer}.
    \end{proof}

We recall that an injective left $R$-module $Q$  is  a {\it cogenerator} in the category of left $R$-modules  if and only if it
cogenerates every simple left $R$-module, or equivalently, $Q$ contains every simple left $R$-module as a submodule (up to isomorphism) (see \cite[Proposition 16.5]{Wisbauer}.
Also,  a cogenerator  $Q$ is called {\it minimal cogenerator}  if $Q\cong E(\oplus_{i\in I} S_i)$, where  $I$ is an index set and $\{S_i~|~i\in I\}$ is a complete set of representatives of the isomorphism classes of simple left $R$-modules. 
Let $R$ and $S$ be two rings. Then an additive contravariant functor $F:R$-mod$\rightarrow$mod-$S$ is called {\it duality}  if it is an equivalence of categories (see \cite[19, Chap. 9]{Wisbauer}]).

A ring $R$ is called a {\it left Morita ring}  if there is an injective cogenerator $_RU$  in $R$-Mod such that for $S={\rm End}(_RU)$, the module  $U_S$ is an injective cogenerator in Mod-$S$ and $R\cong {\rm End}(U_S)$.

\begin{Rem}\label{finite representation type} 
{\rm Let $R$ be  a ring of  finite representation type,  $V:=V_1\oplus\cdots\oplus V_n $,  where   $\{V_1,\cdots,V_n\}$  is a complete set of representatives of the isomorphism classes of finitely
generated indecomposable left $R$-modules. We set   $Q= \bigoplus_{i=1}^mE(S_i)$, where  $\{S_i~|1\leq i\leq m\}$ is a complete set of representatives of the isomorphism classes of simple left $R$-modules. Clearly, the functor ring $T=\widehat{E}nd(V)$ is equal to  $ End_R(V)$,  and in this case, the ring  $T:= End_R(V)$ is called the {\it left Auslander ring}  of $R$. By \cite[Proposition 3.6]{Auslander2},  $T$  is an Artinian ring and so soc$(T_T)$ is an essential submodule  of $T_T$ and $T$ contains only finitely many non-isomorphic types of simple modules.  Also  by \cite[Proposition 46.7]{Wisbauer},  the functor  ${\rm Hom}_R(V,-)$  establishes an equivalence between the category of left $R$-modules and the full subcategory of finitely generated  projective  left $T$-modules,  which preserves and reflects finitely generated  left $R$-modules and  finitely generated indecomposable left $R$-modules  correspond to finitely generated indecomposable projective $T$-modules (see \cite[Proposition  51.7,(5)]{Wisbauer}). In fact,  
each finitely generated indecomposable projective $T$-module is  local and hence it  is the projective cover of  a simple $T$-module. Therefore ${\rm Hom}_R(V,-)$ yields a bijection between a minimal representing set of finitely generated, indecomposable left $R$-modules   and the set of projective covers of non-isomorphic simple left $T$-modules.}
\end{Rem}

\begin{Rem}\label{left Morita}
{\rm By the above
considerations,  $S = End_R(Q)$ is a ring   and since $Q$ is finitely generated, the functor ${\rm Hom}_R(-,Q)$ : $R$-Mod$\rightarrow$  Mod-$S$ is a duality 
 and $S$ is a right Artinian ring. By  duality, we see that $\{Hom_R(V_1,Q),\cdots,Hom_R(V_n,Q)\}$ is a complete set of representatives of the isomorphism classes of finitely generated indecomposable right $S$-modules. It follows  that $S$ is an Artinian ring of finite representation type.  Thus,  $T^\prime =End_S(\bigoplus_{i=1}^n Hom_R(V_i,Q))$ is  the   right  Auslander ring of $S$.  By \cite[Proposition 47.15]{Wisbauer},   the  left Auslander ring  $T$ of $R$ is isomorphic to the  right  Auslander ring  $T^\prime$ of  $S$. Thus,  $R$  is left Morita to $S = End_R(Q)$ (see  \cite[Exercises 52.9 (2)]{Wisbauer}).  Note that the ring  $S = End_R(Q)$ is called the {\it left Morita dual ring of $R$} 
 (see  \cite[Proposition 3.2.]{Kado}}
 \end{Rem}
  
\begin{Rem}\label{duality}
{\rm By the above remarks, and by \cite[Proposition 47.3]{Wisbauer},  we see that  the functor ${\rm Hom}_R(-,Q)$ : $R$-Mod$\rightarrow$  Mod-$S$ is a duality with the inverse duality ${\rm Hom}_S(-,Q)$ : Mod-$S\rightarrow  R$-Mod.
 By \cite[Proposition 24.5]{Anderson-Fuller} for each finitely generated left $R$-module $X$, the lattice of all $R$-submodules of $M$ and the  lattice of all $S$-submodules of
  ${\rm Hom}_R(M,Q)$  are anti-isomorphic. Hence,  for each finitely generated left $R$-module $X$, ${\rm soc(Hom}_R(X,Q))\cong {\rm Hom}_R(top(X),Q)$,
    $top(Hom(X,Q)) = \frac{Hom(X,Q)}{Rad(Hom(X,Q))}\cong Hom(soc(X), Q))$ as $S$-modules (see \cite[24. Exercises, Exercises 6]{Anderson-Fuller}). } 
  \end{Rem}

The following summarizes all the facts presented above on rings of finite representation  type.
\begin{Pro}\label{equivalence}
Let $R$ be a ring of finite representation type,  $U:=U_1\oplus\cdots\oplus U_n $,  where   $\{U_1,\ldots,U_n\}$ is   a complete set of representatives of the isomorphism classes of finitely generated indecomposable left $R$-modules,  $Q\cong E(S_1)\oplus\cdots\oplus E(S_m)$, where  $\{S_i~|~1\leq i\leq m\}$ is a complete set of representatives of the isomorphism classes of simple left $R$-modules,  $T:=End_R(U)$ and $S= End_R(Q)$. Then \vspace*{0.2cm} \\
\indent  {\rm (a)}  The functor  ${\rm Hom}_R(U,-)$  establishes an equivalence between the category of left  \indent \indent $R$-modules and the full subcategory of finitely generated  projective  left $T$-modules.  \\
\indent  {\rm (b)}  The functor ${\rm Hom}_R(-,Q)$ : $R$-Mod$\rightarrow$  Mod-$S$ is a duality  with the inverse duality \indent \indent ${\rm Hom}_S(-,Q) : $Mod-$S\rightarrow  R$-Mod. \\
\indent  {\rm (c)}   $S$ is a right Artinian ring and  $\{Hom_R(U_1,Q),\ldots,Hom_R(U_n,Q)\}$ is a complete \indent \indent set of  representatives of the isomorphism classes of finitely generated indecompos-  \indent \indent able right  $S$-modules.  \\
\indent  {\rm (d)} The  left Auslander ring  $T$ of $R$ is isomorphic to the  right  Auslander ring  $T^\prime$ of  $S$ \indent\indent   (thus $R$  is left Morita to $S = End_R(Q)$). \\
\indent  {\rm (e)}  If  $P$  is  an indecomposable projective   left $T$-module, then  $P\cong Hom_R(U,M)$  for  \indent  \indent some indecomposable left  $R$-module $M$.\\
\indent  {\rm (f)}   An indecomposable left $R$-module $M$  has a square-free  socle if and only if   the left  \indent\indent  $T$-module $Hom_R(U,M)$  has a square-free  socle.  \\
\indent  {\rm (g)}   An indecomposable left $R$-module $M$  has a simple socle if and only if the left $T$-  \indent\indent module     $Hom_R(U,M)$  has a simple socle.  \\
\indent  {\rm (h)}  If  $Y$  is a finitely generated (indecomposable)  right $S$-module with square-free    socle,     \indent\indent then  left $R$-module  ${\rm Hom}_S(Y,Q)$  has a square-free  top.  \\
\indent  {\rm (i)}  If  $Y$  is an indecomposable  right $S$-module with simple   socle,   then   the left $R$-module \indent\indent  ${\rm Hom}_S(Y,Q)$  has a simple  top.  \\
\indent  {\rm (j)}  For each finitely generated  left $R$-module $X$, as $S$-modules,  we have:
 $$soc(Hom_R(X,Q))\cong  Hom_R(top(X),Q).$$
\indent  {\rm (k)}  For each finitely generated  left $R$-module $X$, as $S$-modules,  we have:
 $$top(Hom(X,Q)) = \frac{Hom(X,Q)}{Rad(Hom(X,Q))}\cong Hom(soc(X), Q)).$$
\end{Pro}
\begin{proof}  The facts (a), (b), (c), (d), (e), (f)  and (g)  are optioned  by  Proposition \ref{Projective cover}, and   the facts (h), (i), (j), (k) are optioned  by Remarks \ref{finite representation type},   \ref{left Morita}  and \ref{duality}.
\end{proof}

%77777777777777777777777777777777777777777777777777777777 
%Characterizations of left K{\"o}the rings and K{\"o}the rings%%%%%% 
%77777777777777777777777777777777777777777777777777777777
 
\section{\bf  Characterizations of left co-K{\"o}the rings} 

  Let $M$ be an $R$-module.  We recall that  a submodule $N$  of $M$ is called a {\it $RD$-pure submodule}  of $M$ if  $N\cap rM=rN$ for all
$r\in R$. It  is trivial to verify that every direct summand of $M$ is a $RD$-pure submodule of $M$ (we note that in \cite{chase}, Chase used the term ``pure submodule" insead of ``$RD$-purity"). 

The following two lemmas  play  an important role in this paper.

\begin{Lem}\label{chase}\textup {(Chase \cite[Theorem 3.1]{chase})}
Let $R$ be a ring, and $A$  an infinite set of cardinality $\zeta$ where
$\zeta\geq card(R)$. Set $M =\prod_{\alpha\in A}R^{(\alpha)}$, where $R^{(\alpha)}\cong R$ is a left $R$-module. Suppose
that $M$ is an $RD$-pure submodule of a left $R$-module of the form $N=\bigoplus_\beta N_\beta$, where
each $N_\beta$ is generated by a subset of cardinality less than or equal to $\zeta$. Then $R$
must satisfy the descending chain condition (DCC) on principal right ideals.
 \end{Lem}

\begin{Lem}\label{Zimmermann} \textup {(B. Zimmermann-Huisgen-W. Zimmermann \cite[Page 2]{Zimmermann & Zimmermann})}
For a ring $R$ the following statments are equivalent:  \vspace*{0.2cm} \\
\noindent {{\rm (1)}} Each left $R$-module is a direct sum of finitely generated modules.\\
\noindent {{\rm (2)}} There exists a cardinal number $\aleph$ such that each left  $R$-module is a direct sum of  \indent   ${\aleph}{-\rm generated ~modules.}$\\
\noindent {{\rm (3)}} Each left  $R$-module is a direct sum of indecomposable modules ($R~is~left~pure~semisimple$).
\end{Lem}

The following theorem characterize the rings over which every injective left module is a direct
sum of square-free modules.
 
\begin{The}\label{is left Noetherian} 
The following statements are equivalent for a ring  $R$: \vspace*{0.2cm} \\
\indent {{\rm (1)}} Every injective left $R$-module  is a direct sum  of square-free modules.\\
\indent {{\rm (2)}} Every injective left $R$-module is a direct sum of modules with cyclic socle.\\
\indent {{\rm (3)}} $R$ is left Noetherian.
\end{The}
   
\begin{proof} Let $M$ be an $R$-module, and ${\rm soc}(M) = \bigoplus_{\lambda \in \Lambda} S_{\lambda}$, where $S_{\lambda}$'s are simple modules. Then the cardinal number ${\rm card}(\Lambda)$ is denoted by $c(M)$. This cardinal number is uniquely determined.

(1) $\Rightarrow$ (3). Let $N$ be an injective square-free left $R$-module. The simple submodules of $N$ are non-isomorphic to each other and $c(N) \leqslant \kappa $, where $\kappa$ is the cardinal number of the set of left maximal ideals of $R$. Thus, every injective left $R$-module is a direct sum of modules $\lbrace N_{\gamma}\rbrace_{\gamma \in \Gamma}$ with $c(N_{\gamma}) \leq \kappa$. So by \cite[Proposition 27.5]{Wisbauer}, $R$ is left Neotherian.

(3) $\Rightarrow$ (1). It  is clear by \cite[Proposition 27.5]{Wisbauer}, and the fact that every indecomposable injective R-module is uniform, and hence square-free.

(2) $\Rightarrow$ (3). Let $N$ be an injective left $R$-module with cyclic socle. Set $k = {\rm card}(R)$, then $c(N) \leq {\rm card} (soc(N)) \leq {\rm card}(R)$, since $soc(N)$ is cyclic. Thus, every injective left $R$-module is a direct sum of modules $\lbrace N_{\gamma}\rbrace_{\gamma \in \Gamma}$ with $c(N_{\gamma}) \leq {\rm card}(R)$. So by \cite[Proposition 27.5]{Wisbauer}, $R$ is left Neotherian.

(3) $\Rightarrow$ (2).  It is clear by \cite[Proposition 27.5]{Wisbauer}, and the fact that every indecomposable injective $R$-module is uniform, and so it has a uniform socle. As a result of, it has a simple (and hence cyclic) socle. 
\end{proof}

We recall that a  ring $R$  is said to be of {\it left bounded (representation) type}, if it is  left Artinian and there is a finite upper bound for the lengths of the finitely
generated indecomposable modules in $R$-Mod.  By \cite[Proposition 54.3]{Wisbauer}, $R$  is of left bounded representation type if and only if $R$ is of finite representation type.

The following theorem characterize the rings over which every left module is a direct sum of square-free modules. It is one of the most important and key results of this section and plays an essential role in characterization of left co-K{\"o}the rings.
\begin{The}\label{square-free module} 
The following statements are equivalent for a ring  $R$: \vspace*{0.2cm} \\
\indent {{\rm (1)}} Every left $R$-module is a direct sum of square-free modules.\\
\indent {{\rm (2)}} $R$ is a  left Artinian ring,  and   every  left $R$-module is a direct sum of modules  with \indent \indent square-free socle.\\
\indent {{\rm (3)}} $R$ is of finite representation type,  and  every (finitely generated)  indecomposable  left \indent \indent  $R$-module has a square-free socle.
\end{The}
\begin{proof}
(1) $\Rightarrow$ (2). 
First we prove that for every (minimal) prime ideal $\mathcal{P}$  of $R$ the ring  $R / \mathcal{P}$  is a simple Artinian ring.  For to see this, let $\mathcal{P}$ be a minimal prime ideal of $R$. 
Then by Theorem \ref{is left Noetherian}, $R$ is left Noetherian and so    $R / \mathcal{P}$ is a prime left  Goldie ring. Set $\bar{R} := R / \mathcal{P}$. It's easy to see that every left $\bar{R}$-module is a direct sum of square-free modules. 
Suppose that  $A$ is  an infinite set of cardinality $\zeta$ where
$\zeta\geq card(E(_{\bar{R}}\bar{R}))$,  and we   set $M =\prod_{\alpha\in A}{\bar{R}}^{(\alpha)}$, where $\bar{R}^{(\alpha)}\cong \bar{R}$ as  left $\bar{R}$-modules. By assumption, $M$ 
 is a direct sum of square-free $\bar{R}$-modules $\{K_{\lambda}\}_{\lambda \in \Lambda}$, where $\Lambda$ is a set. By \cite[Proposition 7.8 and 7.9]{Goodearl},
 every left $\bar{R}$-module $K_{\lambda}$ is torsion-free. We claim that each of $K_{\lambda}$ is a uniform $\bar{R}$-module; for if not, let $K := K_{\lambda}$ be non-uniform, then
 $\bar{R}x \bigoplus \bar{R}y \subseteq K$ for some $0 \neq x,~y \in K$. Since $\bar{R}$ has finite rank, $\bar{R}x \bigoplus \bar{R}y$ has finite rank $n \geqslant 2$. But $\bar{R}x \bigoplus \bar{R}y$ is torsion-free, then by
  \cite[Proposition 7.24]{Goodearl}, it has an essential submodule isomorphic to $\bigoplus^{n}U$ where $U$ is a uniform left ideal of $\bar{R}$. This contradicts the square-freeness property of $K$,
   hence $K_{\lambda}$'s are uniform. Thus,  by \cite[Proposition 7.18]{Goodearl}, each $K_{\lambda}$ has an essential submodule $B$ which is isomorphic to a left ideal $I$ of $\bar{R}$. It follows that
\begin{center}
$K_{\lambda} \subseteq E(K_{\lambda}) = E(B) \cong E(I) \subseteq E(\bar{R}).$
\end{center}
Therefore,  each $K_{\lambda}$ is generated by a subset of cardinality less than or equal to $\zeta$, and so   by Lemma \ref{chase}, we conclude that $\bar{R}$ has DCC on principal right ideals. Hence,  $soc_{r}(\bar{R}) \neq 0$. Since $\bar{R}$ is a prime ring, ${\rm soc}_{l}(\bar{R}) = {\rm soc}_{r}(\bar{R})$, and so by \cite[Corollary 7.16]{Goodearl}, $\bar{R}$ is a simple Artinian ring.

 Now we show that,   in fact the ring $R$ left Artinian. Assume that $N(R)$ is the prime radical of $R$. Clearly, every $R/N(R)$-module is a direct sum of square-free $R/N(R)$-modules. Thus,  by Theorem  \ref{is left Noetherian}, $\tilde{R} = R/N(R)$ is a semiprime left Goldie ring and every (minimal) prime ideal of $R/N(R)$ is maximal. So by \cite[Proposition 7.1]{Goodearl}, there exist only finitely many distinct minimal prime,  hence maximal, ideals of $\tilde{R}$. This implies that there exist only finitely many distinct minimal prime,  hence maximal, ideals $\mathcal{P}_{1}, \ldots,\mathcal{P}_n$ of $R$. Now by \cite[Proposition 9.13]{Wisbauer},
 \begin{center}
$R/N(R)\cong R/\mathcal{P}_{1} \bigoplus \cdots \bigoplus R/\mathcal{P}_{n}$
\end{center}
Since each $R/\mathcal{P}_{i}$ is a simple Artinian ring,  $R/N(R)$ is a semisimple Artinian ring. Furthermore, by Theorem \ref{is left Noetherian}, $R$ is left Noetherian, and so $N(R)^{m} = 0$, for some $m \in \mathbb{N}$. Moreover, each $N(R)^{i}/N(R)^{i+1}$ is a finitely generated $R$-module and hence a finitely generated $R/N(R)$-module. Thus, each $N(R)^{i}/N(R)^{i+1}$ is Artinian for $i = 0, \cdots, n$ and so it follows that $R$ is a left Artinian ring.

(2) $\Rightarrow$ (3). 
Since $R$ is a left Artinian ring, there are only finitely many non-isomorphic simple $R$-modules by \cite[Proposition 32.4]{Wisbauer}. Let $W=\bigoplus_{i=1}^k S_i$ be a direct sum of all non-isomorphic simple $R$-modules. Now let $K$ be a  left $R$-module with square-free socle. Then ${\rm {\rm soc}}(K) \hookrightarrow W$  and  ${\rm soc}(K) \subseteq_{e} K$, since $R$ is left Artinian. Hence
\begin{center}
$K \subseteq E(K) = E({\rm {\rm soc}}(K))  \hookrightarrow E(W)~~~~~~(*)$
\end{center}
Let  ${\aleph}=Card(E(W))$,    then   every left $R$-module is a direct sum of ${\aleph}$-generated    $R$-modules and so,   by Lemma \ref{Zimmermann}, $R$ is a  left pure semisimple ring.
 Since each $S_i$ is uniform,  $E(S_i)$ is an indecomposable finitely generated  module and hence ${\rm length} (E(S_i))<  \infty$ for $i=1,\cdots,~k$.  It follows that  ${\rm length}(E(W))<  \infty$ and  we set ${\rm length}(E(W))=m$.  Combining this fact with Inequality $(*)$, we obtain that for each finitely generated  indecomposable  module $K$,  
   $${\rm length} (K) \leq{\rm length} (E(K))\leq {\rm length} (E(W))=m.$$ 
   
   Therefore, $_RR$ is of bounded representation type and so by \cite[Proposition 54.3]{Wisbauer}, $R$ is of  finite representation type  and  every (finitely generated)  indecomposable  left $R$-module has a square-free 
     socle.
     
     (3) $\Rightarrow$ (1).  It is direct.
         \end{proof}   

 Next, we give a new characterization of  left pure semisimple rings as following:

\begin{The}\label{left pure semisimple}
The following are equivalent for a ring $R$: \vspace*{0.2cm} \\
\indent {{\rm (1)}} Every nonzero left $R$-module is a direct sum of modules with nonzero finitely  gene- \indent\indent rated   socle.\\
\indent {{\rm (2)}} $R$ is a left pure semisimple ring.
\end{The}
\begin{proof}
(1) $\Rightarrow$ (2). Assume that $M$ is a left $R$-module. Then  $M=\bigoplus_{\lambda \in \Lambda} K_\lambda$,  where  $\Lambda$ is an index set and   for each ${\lambda \in \Lambda}$,   $K_\lambda\leq M$ and    ${\rm soc}(K_\lambda) \neq 0$,  and it has   finite length.  Since every nonzero submodule of $K_\lambda$ is also  a direct sum of modules with nonzero finitely generated socle,   every  nonzero submodule of $K_\lambda$ contains a simple module.  It follows that  ${\rm soc}(K_\lambda) \leq_e K_\lambda$  for each $\lambda \in \Lambda$. Thus \\
  $K_\lambda \leq  E(soc(K_\lambda))\cong E(\bigoplus_{i=1}^nR/P_i)$,  where  $n\in\Bbb{N}$ and  $P_i$'s are maximal left ideals of 
$R$.

Let ${\aleph}=Card(E)$,  where  $E=\bigoplus_{\Gamma}  E(N_\gamma)$,  and    $\{N_\gamma~|~\gamma\in\Gamma\}$ is a mutually non-isomorphic finitely generated semisimple left $R$-modules.
  Then   every left $R$-module is a direct sum of ${\aleph}$-generated    $R$-modules and so,   by Lemma \ref{Zimmermann}, $R$ is left pure semisimple.

(2) $\Rightarrow$ (1).  It  is  by the fact that a left pure semisimple ring is left Artinian.
\end{proof}

We are now in a position to present the properties and characterizations  of left co-K{\"o}the rings. In fact, our 
theorem characterizes completely those rings of  finite representation type for which any indecomposable module has  a cyclic  (essential) socle, and describes the possible indecomposable modules.

 \begin{The}\label{Theorem}\label{left co-Kothe}
The following conditions are equivalent for a ring  $R$:\vspace*{0.2cm} \\
\indent {{\rm (1)}}   $R$ is a left  co-K{\"o}the  ring.\\
\indent {{\rm (2)}}  Every  nonzero left $R$-module is a direct sum of  modules with nonzero top and  \indent\indent cyclic essential socle.\\
\indent {{\rm (3)}} $R$  is  left  perfect  and every left  $R$-module is a direct sum of  modules with  cyclic  \indent\indent (essential) socle.\\
\indent {{\rm (4)}}  $R$ is  left Artinian  and  every   left $R$-module  is a direct sum of modules with cyclic   \indent\indent  (essential)  socle\\
\indent {{\rm (5)}} $R$ is  left pure semisimple  and  every   left $R$-module  is a direct sum of modules with \indent\indent cyclic (essential) socle.\\
\indent {{\rm (6)}}  $R$ is  of finite representation type and  every  $($finitely generated$)$ indecomposable left  \indent\indent $R$-module  has a cyclic  (essential) socle.\\
 \indent {{\rm (7)}}   $R$ is of finite representation type  and  $soc(U)\hookrightarrow R/J$ for each  indecomposable left \indent\indent  $R$-module $U$.\\
\indent {{\rm (8)}}   $R$ is of finite representation type,   $R\cong \oplus_{i=1}^n(Re_i)^{(t_i)}$, where $\{Re_1,\cdots, Re_n\}$   is \indent\indent a   complete  set of representatives of the isomorphism classes   of 
indecomposable   \indent\indent   projective   left  $R$-modules with  $n,~t_i\in\Bbb{N}$,  and  for each  indecomposable left $R$-module  \indent\indent $U$,  $soc(U) \cong (Re_1/Je_1)^{(u_1)}\oplus\cdots\oplus (Re_n/Je_n)^{(u_n)}$     for some
 $u_i\in\Bbb{N}\cup \{0\}$,  where \indent\indent  $u_i\leq t_i$   for all $i=1,\cdots,n.$
\end{The}
\begin{proof}
(1) $\Rightarrow$ (2). By Theorem \ref{left pure semisimple},  $R$ is left pure semisimple, so by \cite[Theorem 4.4]{chase}, $R$ is a  left Artinian ring and  every left $R$-module is a direct sum of  finitely generated indecomposable  left $R$-modules with  cyclic socle.     Clearly every finitely generated indecomposable module has a maximal submodule and so  it has a nonzero top. 

(1) $\Rightarrow$ (3), (1) $\Rightarrow$ (4) and (1) $\Rightarrow$ (5). By Theorem \ref{left pure semisimple}, $R$ is a left pure semisimple. So $R$ is left Artinian and hence left perfect.

(2) $\Rightarrow$ (1), (3) $\Rightarrow$ (1), (4) $\Rightarrow$ (1) and (5) $\Rightarrow$ (1). By \cite[Proposition 21.2 (3)]{Wisbauer}, every module with essential socle has a nonzero sacle.

(5) $\Rightarrow$ (6).  Since $R$ is  left pure semisimple,  by \cite[Theorem 4.4]{chase}, $R$ is a  left Artinian ring.  
Thus,   there are only finitely many non-isomorphic simple $R$-modules by \cite[Proposition 32.4]{Wisbauer}. Set $W=\bigoplus_{i=1}^k S_i$ is a direct sum of all non-isomorphic simple $R$-modules. Since each $S_i$ is uniform,  $E(S_i)$  and is an   indecomposable finitely generated  module and hence $length(E(S_i))<  \infty$ for $i=1,\cdots,k$. It follows  that  ${length(E(R/J))<\infty}$ and we set  $m=length(E(R/J))$.

Now let $K$ be a  left $R$-module with cyclic  socle. Then  ${\rm soc}(K) \subseteq_{e} K$, since $R$ is left Artinian and $K \subseteq E(K) = E(soc(K))$.  Since  $soc(K)$ is a  cyclic semisimple $R/J$-module,  $soc(K)$ is isomorphic to a direct summand of    $R/J$. 
   It follows that  $length(K)\leq length(E(soc(K)) \leq  length(E(R/J))=m$.    Therefore, $_RR$ is of bounded representation type and so by \cite[Proposition 54.3]{Wisbauer}, $R$ has finite representation type  and  every (finitely generated)  indecomposable  left $R$-module has a cyclic   socle.
   
   (6) $\Rightarrow$ (7).  Since $R/J$ is a semisimple module,   it is clear that if $soc(U)\hookrightarrow R/J$, then  $soc(U)$ is isomorphic to a direct summand of $R/J$ and so it is a cyclic module.   
   
   (7) $\Rightarrow$ (8). It is by Proposition \ref{Projective cover} (a) and (b).
       
   (8) $\Rightarrow$ (1).  It is clear  by Proposition \ref{Projective cover} (a).
 \end{proof}

\section{\bf Characterizations of strongly left  co-K{\"o}the rings}
In this section, we will to characterize strongly  left co-K{\"o}the rings. Also, we give more characterizations of a left co-K{\"o}the ring $R$, whenever $R$ is left quasi-duo. 

We recall that a semi-perfect ring $R$ is called {\it left}  (resp., {\it right})  {\it $QF$-2}  if every indecomposable projective left
(resp., right) $R$-module has a simple essential socle (see \cite{Fuller4}). 
\begin{Defs}
{\rm We will say that a semi-perfect ring $R$  is a\\
{\it - left } (resp., {\it right})  {\it co-$QF$-2  ring} if every indecomposable projective left (resp., right)  \indent  $R$-module has a simple  top.\\
 {\it  -  co-$QF$-2  ring} if $R$ is both a  left  and a   right  co-$QF$-2  ring.\\
 {\it - generalized left}  (resp., {\it right})  {\it co-$QF$-2  ring} if every indecomposable projective left (resp., \indent  right)  $R$-module $P$ has a square-free   top.\\  
{\it - generalized   co-$QF$-2  ring} if $R$ is a generalized left  and a generalized  right  co-$QF$-2  ring.}
\end{Defs}

We are now in a position to present several  properties and characterizations  for strongly left  co-K{\"o}the  rings. In fact, 
the next  theorem characterizes completely those rings of  finite representation type for which any  $($finitely generated$)$ indecomposable left  module  has a square-free  (cyclic)  socle.
Moreover,  the theorem characterizes  those rings over which every left module is a direct sum of square-free modules. Also, 
this characterization  makes an interesting connection between the three concepts, strongly left co-K{\"o}the rings, generalized left   $QF$-2 rings and generalized right   co-$QF$-2  rings.

\begin{The}\label{Theorem}\label{strongly left co-Kothe}
The following conditions are equivalent for a ring  $R$:\vspace*{0.2cm} \\
\indent {{\rm (1)}}   $R$ is a strongly left  co-K{\"o}the  ring.\\
\indent {{\rm (2)}}  Every left $R$-module is a direct sum of square-free modules.\\
\indent {{\rm (3)}}  Every  nonzero left $R$-module is a direct sum of  modules with nonzero top and  \indent\indent square-free  (cyclic)  socle.\\
\indent {{\rm (4)}} $R$  is  left  perfect  and every left  $R$-module is a direct sum of  modules with   square-free  \indent\indent (cyclic) socle.\\
\indent {{\rm (5)}}  $R$ is  left Artinian  and  every   left $R$-module  is a direct sum of modules with square- \indent\indent  free  (cyclic)  socle.\\
\indent {{\rm (6)}} $R$ is  left pure semisimple  and  every   left $R$-module  is a direct sum of modules with   \indent\indent  square-free (cyclic)   socle.\\
\indent {{\rm (7)}}  $R$ is  of finite representation type and  every  $($finitely generated$)$ indecomposable left  \indent\indent $R$-module  has a square-free  (cyclic)  socle.\\
\indent {{\rm (8)}}  $R$ is  of finite  representation type  and  the left Morita dual ring of $R$  is  a strongly  \indent\indent right K{\"o}the  ring.\\  
\indent {{\rm (9)}} $R$ is  of finite  representation type and  the left Auslander ring  of $R$ is a  generalized \indent \indent left   $QF$-2 ring \\
\indent {{\rm (10)}}  $R$ is  of finite representation  type  and the left  Auslander ring  of $R$  is a  generalized \indent\indent right co-$QF$-2 ring.\\  
 \indent {{\rm (11)}}   $R$ is of finite representation type  and   for each  indecomposable left   $R$-module $U$, 
 \indent\indent  $soc(U)\hookrightarrow  (Re_1/Je_1)\oplus\cdots\oplus (Re_n/Je_n),$  where  $e_1,\cdots, e_n$ is a basic set of primitive \indent\indent idempotents for $R$.
\end{The}
\begin{proof}
(1) $\Rightarrow$ (2). Assume that $M$ is a left $R$-module. Then  $M=\bigoplus_{\lambda \in \Lambda} K_\lambda$,  where  $\Lambda$ is an index set and   for each ${\lambda \in \Lambda}$,   $K_\lambda\leq M$ and    ${\rm soc}(K_\lambda) \neq 0$. Since every nonzero submodule of $K_\lambda$ is also  a direct sum of modules with nonzero cyclic socle,   every  nonzero submodule of $K_\lambda$ contains a simple module.  It follows that  ${\rm soc}(K_\lambda) \leq_e K_\lambda$  for each $\lambda \in \Lambda$. Thus, $K_\lambda$ is square-free for each $\lambda \in \Lambda$.  

(2) $\Leftrightarrow$ (5) $\Leftrightarrow$  (7).  By Theorem \ref{square-free module}.

    (7) $\Rightarrow$ (6).  It is clear, since by \cite[Theorem 4.4]{chase},  any ring  of finite  representation type  is  left  (right) pure semisimple.

(6) $\Rightarrow$ (5). It is clear, since by  \cite[Theorem 4.4]{chase},    every left pure semisimple  ring is   a left Artinian ring.

(5) $\Rightarrow$ (4). It is clear, since by \cite[Corollary 28.8]{Anderson-Fuller},  every left Artinian ring is a  left prefect ring. 

(4) $\Rightarrow$ (3). The proof  is direct.

(3) $\Rightarrow$ (1). By Theorem \ref{square-free module}, $R$ is a left pure semisimple ring, and so  by \cite[Theorem 4.4]{chase}, $R$   a left Artinian ring. Thus,  by \cite[Corollary 28.8]{Anderson-Fuller},  $R$ is a  left prefect ring, and so
 by Proposition \ref{Projective cover} (c), if $M$ is a finitely generated left $R$-module with square-free socle, then $soc(M)$ is cyclic. Thus, $R$ is a strongly left  co-K{\"o}the  ring.

In the next steps of proof, for the finite representation type ring $R$, by  Proposition \ref{equivalence},  we have the following assumptions.

\noindent (i) The left Auslander ring  of $R$ is  $T={\rm End}_R(U)$, where  $U=U_1\oplus\cdots\oplus U_n $ and  $\{U_1,\cdots,U_n\}$ \indent is   a complete set of representatives of the isomorphism classes of finitely
generated inde- \indent composable left $R$-modules.\\
(ii)   The  left Morita dual ring of $R$  is     $S=End(Q)$, where  $Q\cong E(S_1)\oplus\cdots\oplus E(S_m)$, where  \indent $\{S_i~|~1\leq i\leq m\}$ is a complete  set  of representatives of the isomorphism 
classes of  \indent simple left $R$-modules.\\
(iii)   $\{Hom_R(U_1,Q),\cdots,Hom_R(U_n,Q)\}$ is a complete  set of  representatives of the isomor \indent phism classes of finitely generated indecomposable  right  $S$-modules.

(7) $\Rightarrow$ (8).  By Proposition \ref{equivalence} (c) and  \cite[Theorem 3.4]{Behboodi3}, it is enough to show that each  of indecomposable right $S$-modules $Hom_R(U_1, Q),\cdots, Hom_R(U_n, Q)$, has a square-free top. Since  each $U_i$ has a  square-free socle,    by Proposition \ref{equivalence} (j),  each indecomposable right $S$-module ${\rm Hom}_R(U_i,Q)$   has a square-free top.  Thus,   left Morita dual ring $S=End(Q)$  of $R$  is  a strongly right  K{\"o}the ring.
     
   (8) $\Rightarrow$ (7). By \cite[Theorem 3.4]{Behboodi3}.

(7) $\Rightarrow$ (9).  Since every indecomposable left $R$-module has a  square-free   socle, by Proposition \ref{equivalence}, every indecomposable projective left $T$-module has a square-free   socle. Thus,  $T$ is a generalized left $QF$-2 ring,  by our definition. 

(9) $\Rightarrow$  (7).  By Proposition \ref{equivalence} (d),  $T\cong T^\prime$,  where   $T^\prime$ is the right  Auslander ring  of $S$. By Proposition \ref{equivalence} (c) and (e), it is enough to show that $Hom(U_{i} , Q)$ has a square-free top. By Proposition \ref{equivalence} (k), we have $top(Hom(U_{i} , Q)) \cong Hom_{R}(soc(U_i), Q)$ and since  each $top(Hom(U_{i} , Q))$  is   square-free, we conclude that  $soc(U_{i})$  is  also  square-free for each $i$. Thus,  
 $R$ is of finite  representation type and  every  $($finitely generated$)$ indecomposable  left $R$-module has  a square-free   socle.
 
(6) $\Rightarrow$  (8).  
By Proposition \ref{equivalence} (d),  $T\cong T^\prime$,  where  $T^\prime$ is the right  Auslander ring  of $S$. Since every indecomposable right $S$-module    has a square-free top,   every finitely generated indecomposable projective right $T$-module  has a square-free top, i.e., $T$ is  a generalized right  co-$QF$-2 ring.

(7) $\Leftrightarrow$  (11). It is by Proposition \ref{Projective cover} (b).
\end{proof}
 A ring $R$ is said to be {\it left duo} (resp., {\it  left quasi-duo}) if every left ideal (resp. maximal left ideal) of $R$ is an ideal. Obviously, left duo rings are left
quasi-duo. Other examples of left quasi-duo rings include, for instance, the commutative rings, the local rings, the rings in which every non-unit has a (positive) 
power  that  is a  central element, the endomorphism rings of uniserial modules, and the  power series rings and  the rings of upper triangular matrices over any of the above-mentioned  rings (see  \cite{Yu}).  It
is easy to see that if a ring $R$ is left duo (resp. left quasi-duo), so is any factor ring of $R$.  By a result of  \cite{Yu},  a ring $R$ is left quasi-duo if and only if  $R/J(R)$ is left quasi-duo and  if 
$R$ is left quasi-duo, then $R/J(R)$ is a subdirect product of division rings. The  converse is not true in general (see \cite[Example 5.2]{Lam}), but   the converse is true when $R$ has only a finite number of simple left $R$-modules (up to isomorphisms) (see \cite[Page 252]{Lam}.  Consequently,  a semilocal ring $R$ is left  (right) quasi-duo if and only if  $R/J(R)$  is a direct product of division rings.
%We recall that a ring $R$ is called basic if $R/J$ is a direct sum of division rings and idempotents in $R/J$ can be lifted to $R$. Thus, 
 Note that any basic ring is a left and a right quasi-duo ring.

The following theorem shows that for a left quasi-duo ring $R$ the concepts of   ``left  co-K{\"o}the"   and  ``strongly left co-K{\"o}the"   coincide (this is the dual of \cite[Theorem 3.5]{Behboodi3}).

\begin{The}\label{left quasi duo left co-kothe}
The following conditions 
are equivalent for a  left quasi-duo ring $R$: \vspace*{0.2cm} \\
\indent {{\rm (1)}} $R$  is a left co-K{\"o}the ring.\\
\indent {{\rm (2)}} $R$  is a strongly left co-K{\"o}the ring.\\
\indent {{\rm (3)}}  $R$ is of finite representation type and every indecomposable module has a square-  \indent\indent free socle.\\
\indent {{\rm (4)}}  $R$ is of finite representation type and every indecomposable module has a cyclic  \indent\indent socle.\\
\indent {{\rm (5)}}  $R$ is  of finite  representation type  and  the left Morita dual ring of $R$  is  a   right K{\"o}the \indent\indent  ring.\\  
\indent {{\rm (6)}}  $R$ is  of finite  representation type  and  the left Morita dual ring of $R$  is  a strongly  \indent\indent right K{\"o}the  ring.\\  
\indent {{\rm (7)}} $R$ is  of finite representation  type and  the left Auslander ring  of $R$ is a  generalized  \indent\indent left   $QF$-2 ring \\
\indent {{\rm (8)}}  $R$ is of finite representation  type and the left Auslander ring  of $R$ is  a generalized  \indent\indent right   co-$QF$-2  ring. 
\end{The}
\begin{proof}
(1) $\Rightarrow$ (2). By Theorem \ref{Theorem},  $R$ is  of finite representation type and  every  indecomposable  left $R$-module has a square-free  socle.  Let $M$ be an indecomposable Artinian module. By \cite [Proposition 31.3]{Wisbauer}, ${\rm Soc}(M) \cong \bigoplus_ {j=1}^{m} R/\mathcal{M}_{j}$. Since ${\rm Soc}(M)$ is square-free, for every $i \neq j$, $R/\mathcal{M}_{i}$ and $R/\mathcal{M}_{j}$ are non-isomorphic simple modules. Now by \cite [Proposition 9.13]{Wisbauer}, ${\rm Soc}(M) \cong R/(\bigcap_{i=1}^{m} \mathcal{M}_{j})$, and so it   is cyclic.

(2) $\Rightarrow$ (1). Since $R$ is a left quasi-duo ring,  every  simple left  $R$-module is isomorphic to $R/P,$ where $P$ is a maximal  left ideal of $R$, i.e.,   $R/P$ is a division ring.
 Let $M$ be an indecomposable module with a nonzero cyclic socle. By \cite [Proposition 31.3]{Wisbauer}, ${\rm Soc}(M) \cong R/P_1 \oplus \cdots R/P_n$. We claim that ${\rm Soc}(M)$ is square-free. Suppose not, then at least two of the simple modules $R/P_1 \oplus \cdots R/P_n$  are isomorphic. Let $P_1 = P_2 = P$ so, ${\rm Soc}(M) \cong R/P \oplus R/P \oplus \cdots R/P_n$.
Since ${\rm Soc}(M)$ is cyclic,  every factor module  of ${\rm Soc}(M)$ is  also cyclic. Thus,  $R/P \oplus R/P$ is cylic, but it is an  $R/P$-vector space and $dim_{R/P} R/P \oplus R/P = 2$. So,  it can not a cyclic $R/P$-moule, hence it can not a cyclic $R$-module, and this contradicts our assumption. Then ${\rm Soc}(M)$ is square-free and the  proof is complete.

(5) $\Leftrightarrow$ (6). By the equivalence of (1) and (2). 

(1) $\Leftrightarrow$ (4). By Theorem \ref{left co-Kothe}.

(2) $\Leftrightarrow$ (3), (3) $\Leftrightarrow$ (6), (6) $\Leftrightarrow$ (7) and (7) $\Leftrightarrow$ (8).  All are by  Theorem \ref{Theorem}.
\end{proof}

 \begin{Cor}\label{S is strongly right Kothe}
Let $R$ be a ring  and $S$ is a Morita ring of $R$ :  \vspace*{0.2cm} \\
\indent {{\rm (1)}} Basic ring $A_0$ of $R$ is a stongly left K{\"o}the ring and $S$ is a strongly right~K{\"o}the~ring.\\
\indent {{\rm (2)}} Any indecomposable left $R$-module has square-free top and square-free socle.
\end{Cor}
\begin{proof}
(1) $\Leftrightarrow$ (2). By \cite[Theorem 3.4]{Behboodi3} and Theorem \ref{Theorem}. \end{proof}
\begin{The}\label{Lem}
Let $A$ is  a finite dimensional $K$-algebra. If  $A$ is a strongly left co-K{\"o}the  ring then 
 $A$ is a right K{\"o}the  ring.
\end{The}

\begin{proof}
 By  Theorem  \ref{square-free module},  $A$ is  of finite representation type and  every  indecomposable  left   $A$-module has square-free  socle. Thus every  indecomposable  left $A$-module has squar-free socle and  so it is suffices  to show that  every indecomposable  right $A$-module  cyclic. 
 
 Let  $M$ a finitely generated  left $A$-module and   $(0)=M_0\lneq M_1\lneq ....\lneq M_l=M$ be any composition series for $M$. 
  We choose a fixed ordering  $S_1, ...,~S_n$  of the simple  $A$-modules,  and denote by $(\underline{dim}~M)_i $ the number of composition factors of $M$ isomorphic to $S_i$ , this number is independent of the given composition series (theorem of Jordan-H{\"o}tlder).  In particular, we always have the indecomposable direct summands of the left module $_A A$ , we denote representatives of their isomorphism classes by $P_1,...,P_n $. Since  the series $(0)\leq P_i\leq A$ can be extend  to a  composition series,    the  factor module $top(P_i)=P_i/Rad(P_i)$ is simple. We can assume that     $top(P_i) \cong  S_i$ where 
$1 < i < n$. Thus, $_A A = \bigoplus (P_i)^ {(t_i)}$,   where  $t_i\in N$ for each $i$   and  $(P_i)^ {(t_i)}$ denotes the 
direct sum of $t_i$ copies of $P_i$.  \\
 \indent Now  let $M_1,...,M_m$  be the indecomposable left $A$-modules. 
 For each  module $M_j$,  its dual module  $Hom_K(M_j,K )$   is a right  $A$-module and  $Hom_K(M_1,K ),...,Hom_K(M_m,K )$ are the indecomposable right $A$-modules. Since every  indecomposable left $A$-module has square-free  socle,    $(\underline{dim}~SocM_j)_i =1$  for all $i$,  and so  by    \cite[Proposition 1.5]{Ringel}),    $Hom_K(M_i,K ) $  is cyclic  for all $i$ and so   $A$ is a right K{\"o}the  ring. 
  \end{proof}
\begin{Lem}{\rm({\cite[Theorem 1.6]{Ringel}})}.\label{basic Kawada}
The following conditions now obviously are equivalent for a finite-dimensional algebra $A$:\\
\indent {{\rm (1)}} Basic ring $A_0$ of $R$ is a K{\"o}the algebra.\\
\indent {{\rm (2)}} Any indecomposable $A$-module has square-free top and square-free socle.
\end{Lem}
Thus by \cite[Theorem 3.5]{Behboodi3},  for any basic ring $R$ the concepts of stongly K{\"o}the and K{\"o}the concide, and by Corollary \ref{S is strongly right Kothe}, we have the following result that is generalization of Ringle Theorem:
 \begin{Cor}\label{S morita R is stongly right Kothe}
Let $R$ be a ring and  $S$ is a Morita ring of $R$ such that $S \cong R_R$: \vspace*{0.2cm} \\
\indent {{\rm (1)}} $R$  is a stongly K{\"o}the ring.\\
\indent {{\rm (2)}} Any indecomposable left $R$-module has square-free top and square-free socle.
\end{Cor}

  %55555555555555555555555555555555555555555555555555555555555
\section{\bf  Characterizations of very strongly left co-K{\"o}the rings}
%55555555555555555555555555555555555555555555555555555555555

An  $R$-module $N$ is called {\it uniserial}  if its submodules are linearly ordered
by inclusion. If $_RR$ (resp. $R_R$) is uniserial we call $R$ {\it  left}  ({\it right}) {\it uniserial}. Note that finitely generated uniserial modules  are in particular local module, and local modules are cyclic modules.
A {\it uniserial ring}  is a ring which is both left and right uniserial. Note that  commutative uniserial rings are also known as valuation rings.
In  addition,  we call an $R$-module $M$  {\it serial}  if it is a direct sum of uniserial modules.
The ring $R$ is called  {\it left  (right) serial}  if $_RR$ (resp. $R_R$) is a serial module.
We say $R$ is  {\it serial}  if $R$ is left and right serial.

An Artinian ring $R$ is said to have {\it left  colocal type} if every finitely generated indecomposable left $R$-module has a simple socle. Such rings and algebras have been investigated by several authors including Makino \cite{Makino}, Sumioka \cite{Sumioka1, Sumioka2}, Tachikawa \cite{Tachikawa1, Tachikawa2}, and a special case by Fuller \cite{Fuller3}. 
Artinian serial rings clearly have finite representation type,  as well as right and left colocal type.

Extending modules form a natural class of modules which is more general than the class of injective modules,  but retains many of its desirable properties.  Recall that a module $M$ is called an {\it extending module} (or, {\it $CS$-module}) if every submodule is essential in a direct summand of $M$. Extending modules generalize (quasi-)injective, semisimple, and
uniform modules and have been extensively studied over the last few decades (see \cite{Dung}
for a detailed account of such modules). In \cite[Theorem 1]{Er}, Er  proved that   the rings whose right modules are direct sums of extending modules coincide with  those that have finite representation type and right colocal type.
An $R$-module $N$ is called {\it co-cyclic}  if $N$ has a simple submodule $K$ which is contained in every nonzero submodule of $N$.

The following is our characterization of very strongly left  co-K{\"o}the  rings. In fact, 
the next  theorem characterizes completely those rings of  finite representation type for which any  $($finitely generated$)$ indecomposable left  module   has a simple socle.
Among other  characterizations,  we show  that a ring $R$ is a very strongly left  co-K{\"o}the  ring if and only if every left  $R$-module is a direct sum of extending modules, if and only if,  every left $R$-module is a direct sum of uniform modules. Moreover, we show  that   any very strongly left  co-K{\"o}the  ring $R$  is  an Artinian  left serial  ring.

\begin{The}\label{very strongly left co-Kothe}
The following conditions are equivalent for a ring  $R$:\vspace*{0.2cm} \\
\indent {{\rm (1)}}  $R$ is a very strongly left  co-K{\"o}the  ring.\\
\indent {{\rm (2)}}   Every  left $R$-module is a direct sum of co-cyclic modules.\\
\indent {{\rm (3)}}  Every  nonzero left $R$-module is a direct sum of  modules with nonzero top and  \indent\indent simple socle.\\
\indent {{\rm (4)}}  $R$  is a  left  perfect ring,   and every left  $R$-module is a direct sum of  modules with \indent\indent  simple  socle.\\
\indent {{\rm (5)}}  $R$ is a  left Artinian ring,   and  every   left $R$-module  is a direct sum of modules with \indent\indent simple  socle.\\
\indent {{\rm (6)}} $R$ is a  left pure semisimple ring,   and  every   left $R$-module  is a direct sum of modules  \indent\indent  with simple socle.\\
\indent {{\rm (7)}}  $R$ is  of finite representation type and  every  $($finitely generated$)$ indecomposable left  \indent\indent $R$-module  has a simple socle.\\
\indent {{\rm (8)}} Every left  $R$-module is a direct sum of extending modules.\\
\indent {{\rm (9)}} Every left $R$-module is a direct sum of uniform modules.\\
\indent {{\rm (10)}}  $R$ is  of finite representation  type,   and the  left Morita dual ring of $R$  is  a  very  \indent\indent strongly  right K{\"o}the  ring.\\  
\indent {{\rm (11)}} $R$ is   of finite representation  type,  and  the  left Auslander ring $T$ of $R$ is a   left   \indent\indent $QF$-2 ring \\
\indent {{\rm (12)}}  $R$ is   of finite representation  type,  and the  left Auslander ring  $T$ of $R$ is  a right \indent\indent co-$QF$-2  ring.\\ 
\indent {{\rm (13)}}   $R$ is of finite representation type,   and   for each  indecomposable left   $R$-module
\indent\indent  $U$,  $soc(U) \cong Re_i/Je_i$  where  $e_i \in \{e_1,\cdots, e_n\}$ and $\{e_1,\cdots, e_n\}$ is a basic set of \indent\indent  primitive idempotents for $R$.  

If these assertions hold, then $R$ is  an Artinian  left serial  ring.
\end{The}
\begin{proof} 
The equivalence part (2) to (7) is by Theorem \ref{Theorem}, since the module with simple socle is a module with square-free socle. 

(1) $\Leftrightarrow$ (2). It is clear, since a left module is co-cyclic if and only if it has an essential simple socle.
        
(7)  $\Leftrightarrow$ (8)  $\Rightarrow$  (9). It  is by \cite[Theorem 1]{Er}.

(9) $\Rightarrow$  (8).  It is clear, since every uniform module is extending.

In the next steps of proof, for the finite representation type ring $R$, by  Proposition  \ref{equivalence},  we have the following assumptions.

\noindent (i) The left Auslander ring  of $R$ is  $T={\rm End}_R(U)$, where  $U=U_1\oplus\cdots\oplus U_n $ and  $\{U_1,\cdots,U_n\}$ \indent is   a complete set of representatives of the isomorphism classes of finitely
generated inde- \indent composable left $R$-modules.\\
(ii)   The  left Morita dual ring $R$ is  $S=End(Q)$, where  $Q\cong E(S_1)\oplus\cdots\oplus E(S_m)$, where \indent $\{S_i~|~1\leq i\leq m\}$ is a completet set  of representatives of the isomorphism 
classes of  \indent simple  left $R$-modules.\\
(iii)   $\{Hom_R(U_1,Q),\cdots,Hom_R(U_n,Q)\}$ is a complete  set of  representatives of the isomor-  \indent  phism classes of finitely generated indecomposable  right  $S$-modules.
 
(7) $\Rightarrow$ (10). By Proposition \ref{equivalence} (c) and \cite[Theorem 4.1]{Behboodi3} , it is enough to show that each  of indecomposable right $S$-modules $Hom_R(U_1, Q),\cdots, Hom_R(U_n, Q)$, has a simple top.  Since  each $U_i$ has a simple socle, by Proposition \ref{equivalence} (j),  each indecomposable right $S$-module ${\rm Hom}_R(U_i,Q)$   has a simple top.  Thus,   left Morita dual ring $S=End(Q)$  of $R$  is  a  very  strongly right  K{\"o}the ring.
  
  (10) $\Rightarrow$  (7). By  \cite[Theorem 4.1]{Behboodi3}.

(7) $\Leftrightarrow$ (11).  By Proposition  \ref{equivalence} (g), every indecomposable left $R$-module has a simple   socle if and only if every indecomposable projective left $T$-module has a simple  socle, i.e.,  $T$ is left $QF$-2.

(7) $\Rightarrow$  (12). By Proposition \ref{equivalence}  (b),  $T\cong T^\prime$,  where   $T^\prime$ is the right  Auslander ring  of $S$. Thus by \cite[Corollary 2.5]{Behboodi3},  it is enough to show that   every  finitely generated projective indecomposable right 
$T^\prime$-module has a simple top. On the other hand, by   Proposition \ref{equivalence}  (a) and (g),  it is enough to show that   every indecomposable right $S$-module  $Hom(U_i,Q)$   has a simple top.  
By our hypothesis each  $U_i$  has a simple   socle and hence  $Hom(soc(U_i), Q)) \cong soc(\bigoplus_{k=1}^{m}Hom(S_{i}, Q)) = top(Hom(U_i , Q))$ is simple, and the proof is  completed, and so  $R$ is   of finite 
representation  type and the left Auslander ring  of $R$ is  right co-$QF$-2. 

(12) $\Rightarrow$  (7).  By Proposition \ref{equivalence} (d),  $T\cong T^\prime$,  where   $T^\prime$ is the right  Auslander ring  of $S$. By Proposition \ref{equivalence} (j), we have $top(Hom(U_{i} , Q)) \cong Hom_{R}(soc(U_i), Q)$ and since  each $top(Hom(U_{i} , Q))$  is  simple, we conclude that  $soc(U_{i})$  is  also  simple for each $i$. Thus 
 $R$ is of finite representation  type and  every  $($finitely generated$)$ indecomposable  left $R$-module has a simple  socle.  
       
(7) $\Leftrightarrow$ (13). It  is by  Proposition \ref{Projective cover} (b).

The final statement   {\it ``$R$ is  an Artinian  left serial  ring"}  is also from \cite[Theorem 1]{Er}.
\end{proof}

\begin{Cor}\label{S is very stongly right Kothe}
Let $R$ be a ring and $S$ be the Morita dual ring of $R$. Then the  following statements are equivanlent:  \vspace*{0.2cm} \\
\indent {{\rm (1)}} $R$  is a very strongly left K{\"o}the ring and $S$ is a  very strongly right K{\"o}the ring.\\
\indent {{\rm (2)}} Any indecomposable left $R$-module has simple socle and simple top.\\
\indent {{\rm (3)}} $R$ is Artinian serial ring.
\end{Cor}
\begin{proof}
It is by \cite[Theorem 4.1]{Behboodi3}, \cite[Theorem 55.16]{Wisbauer} and Theorem \ref{very strongly left co-Kothe}. 
\end{proof}
 
 An Artinian ring $R$ is said to have a {\it self}-({\it Morita}) {\it duality} if there is a Morita duality $D$ between $R$-mod, the category of finitely generated left $R$-modules,
and mod-$R$, the category of finitely generated right $R$-modules. Since we are assuming that $R$ is Artinian, Morita \cite{Morita} and Azumaya \cite{Azumaya} have shown:
$R$ has a self-duality $D$ if and only if there is an injective cogenerator $_RE$ of $R$-mod and a ring isomorphism $v: R \rightarrow End(_RE)$ 
(which induces a right $R$-structure on $E$ via $x . r = xv(r)$ for $x \in E$ and $r \in R$), such that the dualities $D$ and  $Hom_R( -, _RE_R)$ are naturally equivalent.

 A module $M$ is called {\it lifting } if every submodule $N$ of $M$ lies above a direct summand of $M$,  i.e., there exists a direct-sum decomposition $M = M_1 \bigoplus M_2$ with $M_1 \subseteq N$ and $N \bigcap M_2$ superfluous in $M_2$. Extending modules generalize (quasi-)injective, semisimple, and uniform modules,  while lifting modules extend semisimple and hollow modules. These modules
have been extensively studied in the last years (see, for instance, \cite{Clark, Dung} for a detailed account on
them). Also, a ring $R$ is called of {\it left local type}  if every finitely generated indecomposable left $R$-module is local (see \cite{Goodearl}).

By \cite[Corollary  5.13]{Kado},  every Artinian serial ring  has self-duality. By using this fact,   the  above theorem  and \cite[Proposition 55.16]{Wisbauer}, we have the following  characterization of (very) strongly co-K{\"o}the rings.
In particular, the next  theorem characterizes completely those left Artinian rings for which any  $($finitely generated$)$ indecomposable left  module   has  a simple socle and  simple top.
\begin{The}\label{Kothe}
The following statements are equivalent for any  ring $R$:\vspace*{0.2cm} \\
\indent {{\rm (1)}} $R$ is a very strongly co-K{\"o}the ring.  \\
\indent {{\rm (2)}} $R$ is a very strongly K{\"o}the ring.  \\
\indent {{\rm (3)}}  Every left and right $R$-module  is a direct sum of local modules.\\
\indent {{\rm (4)}} $R$ is an Artinian serial ring.\\
\indent {{\rm (5)}} Every left $R$-module is serial.\\
\indent {{\rm (6)}} Every right $R$-module is serial.\\
\indent {{\rm (7)}}  $R$ is of finite representation  type and has  local type.\\
\indent {{\rm (8)}}  $R$ is of finite representation  type and has  colocal type.\\
\indent {{\rm (9)}} Every left and right $R$-module is a direct sum of uniform modules.\\
\indent {{\rm (10)}} Every left and right $R$-module is a direct sum of extending modules.\\
\indent {{\rm (11)}} Every left and right $R$-module is a direct sum of lifting modules.\\
\indent {{\rm (12)}}  $R$ is  of finite representation  type  and the left  (right)  Auslander ring   of $R$  is a \indent\indent $QF$-2 ring.\\   
\indent {{\rm (13)}}  $R$ is   of finite representation   type and the left  (right) Auslander ring   of $R$  is a \indent\indent co-$QF$-2 ring.\\
\indent {{\rm (14)}}  $R$ is left Artinian  and  every  $($finitely generated$)$ indecomposable left $R$-module  is \indent\indent  uniserial.\\
\indent {{\rm (15)}} Every left and right $R$-module is a direct sum of  finitely generated   modules with   \indent\indent square-free  top.  \\
\indent {{\rm (16)}} Every left and right $R$-module is a direct sum of  (finitely generated)  indecomposable \indent  \indent modules with   simple  socle and simple top.  \\
\indent {{\rm (17)}}  $R$ is left Artinian and every  $($finitely generated$)$ indecomposable left $R$-module has \indent\indent a simple socle and  simple top.
 \end{The}
\begin{proof}

(1) $\Leftrightarrow$ (4). It  is by Theorem \ref{very strongly left co-Kothe} and \cite[Corollary  2]{Er}.

(2) $\Leftrightarrow$ (4). It  is by  \cite[Theorem 4.7]{Behboodi3}

  (8) $\Leftrightarrow$ (9). It is by \cite[Corollary 2.3]{Pedro}.
   
(2) $\Leftrightarrow$ (3). It  is   by  \cite[Theorem 4.1]{Behboodi3}.

(4) $\Leftrightarrow$ (5)  $\Leftrightarrow$ (6) $\Leftrightarrow$ (16). It  is by  \cite[Proposition 55.16 (1)]{Wisbauer}.

(3) $\Leftrightarrow$ (7) $\Leftrightarrow$ (11)  $\Leftrightarrow$  (12). It  is by  \cite[Theorem 4.1]{Behboodi3}. 

(4) $\Leftrightarrow$ (5) $\Leftrightarrow$ (8) $\Leftrightarrow$ (10)  $\Leftrightarrow$ (13) $\Leftrightarrow$ (14)  $\Leftrightarrow$ (15)  $\Leftrightarrow$ (16). It  is by Theorem \ref{very strongly left co-Kothe}, and the fact that  every Artinian serial ring  has self-duality..

(5) $\Leftrightarrow$ (17). It    is by \cite[Proposition 55.16 (2)]{Wisbauer}.
\end{proof}

Now the following corollary  is obtained immediately.
\begin{Cor}\label{Commutative co-Kothe}
The following conditions are equivalent for a  commutative ring  $R$:\vspace*{0.2cm} \\
\indent {{\rm (1)}} $R$ is a  K{\"o}the  (co-K{\"o}the) ring.  \\
\indent {{\rm (2)}} $R$ is a strongly  K{\"o}the  (co-K{\"o}the) ring. \\
\indent {{\rm (3)}} $R$ is a very strongly  K{\"o}the  (co-K{\"o}the) ring.  \\
\indent {{\rm (4)}} Every $R$-module is a direct sum of local   modules.\\
\indent {{\rm (5)}} Every  $R$-module is a direct sum of cyclic
 modules.\\
\indent {{\rm (6)}} Every $R$-module is a direct sum of uniform modules.\\
\indent {{\rm (7)}} Every $R$-module is a direct sum of uniserial  modules.\\
\indent {{\rm (8)}} Every $R$-module is a direct sum of extending modules.\\
\indent {{\rm (9)}} Every  $R$-module is a direct sum of  square-free  modules. \\
\indent {{\rm (10)}} Every  $R$-module is a direct sum of   indecomposable  modules.\\ 
\indent {{\rm (11)}} Every  $R$-module is a direct sum of  finitely generated  modules. \\
\indent {{\rm (12)}} Every  $R$-module is  isomorphic to a direct sum of ideals of $R$. \\
\indent {{\rm (13)}} $R$ is an Artinian principal ideal ring.\\
\indent {{\rm (14)}} $R$ is an Artinian serial ring.
\end{Cor}
   
By  \cite[Proposition 27.14]{Anderson-Fuller} for any  semi-perfect ring $R$, there exists (uniquely up to isomorphism) a basic ring $A_0$ which is Morita equivalent to $R$. 
In fact, a ring $A_0$ is a {\it basic ring } for $R$ in case $A_0$ is isomorphic to $eRe$ for some basic idempotent $e \in R$.  One can easily see that  any finite product of local rings  is a   semi-perfect basic ring. 
 
We recall that during the years 1962 to 1965,  Kawada \cite{Kawada1,Kawada2,Kawada3} solved  K{\"o}the problem for basic finite-dimensional algebras. 
Kawada characterized completely those finite-dimensional algebras $A$ for which  $soc(K)$ and $top(K)$ are simple for each  finitely generated indecomposable $A$-module $K$,  and describes the possible 
indecomposable modules.  A ring $R$ is called {\it left}  (resp., {\it right})  {\it Kawada}  if any ring Morita equivalent to $R$ is a left  (resp., right) K{\"o}the ring.  By a theorem of Ringel \cite{Ringel} a finite dimensional $K$-algebra
$A$ is a Kawada algebra if and only if the basic ring of $A$ is a K{\"o}the algebra, if and only if,
for each finitely generated indecomposable $A$-module $K$, $soc(K)$ and $top(K)$ are simple.
We recall that the left K{\"o}the property is not a Morita invariant property, and following Ringel \cite{Ringel} 
we say that a ring $R$ is left (resp., right) Kawada if any ring Morita equivalent to $R$ is a left  (resp., right) K{\"o}the ring.
 
Next,   we give the   following characterizations of  rings over which every  left module is a direct sum of local modules and the product of any two maximal  ideals commutes.
\begin{Lem}\label{prime ideals commutes}
The following statements are equivalent for any  ring $R$:\vspace*{0.2cm} \\
\indent {{\rm (1)}}  $R$ is a very strongly left K{\"o}the ring in which the product of any two maximal ideals \indent\indent commutes.\\   
\indent {{\rm (2)}} $R\cong Mat_{n_1}(R_1)\times\cdots\times  Mat_{n_k}(R_k)$, where $k,~n_1,\cdots,n_k\in\Bbb{N}$,   each  $R_i$ is a   local  \indent\indent  Artinian    ring and  each $Mat_{n_i}(R_i)$  is a direct sum of local  left modules for each  \indent\indent  $1\leq i\leq k$. 
\end{Lem}
\begin{proof} (1) $\Rightarrow$ (2). By  \cite[Theorem 4.1]{Behboodi3}, $R$ is of finite representation type.  Thus,  $R$ is an Artinian ring.   Now by \cite[Theorem 3.6] {Tolooei}, 
$R\cong Mat_{n_1}(R_1)\times\cdots\times  Mat_{n_k}(R_k)$, where $k,~n_1,\cdots,n_k\in\Bbb{N}$,  and  each  $R_i$ is a   local   Artinian    ring. Clearly,  each $Mat_{n_i}(R_i)$  is also a  left  K{\"o}the  ring  for each  $1\leq i\leq k$.

(2) $\Rightarrow$ (1).  It is clear.
\end{proof}

Next,   we give the   following characterizations of  Kawada rings  for which the product of any two maximal  ideals commutes.
\begin{The}\label{commutes-Kawada}
The following statements are equivalent for any  ring $R$: \vspace*{0.2cm} \\
\indent {{\rm (1)}}  $R$ is a left Kawada ring in which the product of any two maximal ideals commutes.\\
\indent {{\rm (2)}}  $R$ is  of finite representation   type  and  the  Auslander ring $T$ of $A_0$ is a  right $QF$-2  \indent\indent  ring.\\
\indent {{\rm (3)}} The basic ring $A_0$ of  $R$ is a finite direct product of  local (very strongly) left  K{\"o}the \indent\indent rings.\\
\indent {{\rm (4)}} $ R \cong Mat_{n_1}(R_1)\times\cdots\times  Mat_{n_k}(R_k)$, where $k,~n_1,\cdots,n_k\in\Bbb{N}$  and  each $R_i$ is a  \indent\indent  local  (very strongly) left K{\"o}the  ring  for  $1\leq i\leq k$. 

Consequently, a Kawada ring is an  Artinian principal  ideal ring 
if and only if its
maximal ideals commute.
\end{The}
\begin{proof} The proof is by Lemma \ref{prime ideals commutes},  \cite[Proposition 2.7]{Behboodi3}  and  by  \cite[Theorem 3.5]{Behboodi3}.
\end{proof}

\begin{Cor} {\rm(See also \cite[Proposition 4.6]{Behboodi3}})\label{Cor prime ideals commutes-Abelian}
\it{Let $R$ be  a  ring  in which all the idempotent  elements of $R$ are central.  Then the following statements are equivalent:  \vspace*{0.2cm} \\
\indent {{\rm (1)}} $R$ is a K{\"o}the  (co-K{\"o}the)  ring.\\
\indent {{\rm (2)}} $R$ is a strongly K{\"o}the  (co-K{\"o}the) ring.\\
\indent {{\rm (3)}} $R$ is a very strongly K{\"o}the  (co-K{\"o}the)ring.\\
\indent {{\rm (4)}} $R$ is isomorphic to a finite product of Artinian uniserial rings.}
 \end{Cor}  
\begin{proof}
It is clear that for any local ring, strongly (co-)K{\"o}the concept is equivalent to very strongly (co-)K{\"o}the . Then proof is by \cite[Proposition 3]{Jebrel},   Theorem \ref{prime ideals commutes}, Theorem \ref{Kothe}, Theorem \ref{left quasi duo left co-kothe} and the fact that in any Artinian local ring the Jacobson radical is the unique maximal (prime)  ideal.
\end{proof}

 \begin{Pro}\label{L33}
 The following statements are equivalent for any  local ring $R$: \vspace*{0.2cm} \\
 \indent {{\rm (1)}} Every cyclic left $R$-module is a direct sum of square-free modules.\\
 \indent {{\rm (2)}} Every cyclic left $R$-module is a direct sum of uniform modules.\\
 \indent {{\rm (3)}} $R$ is a left uniserial ring.
\end{Pro}
\begin{proof}
If $R$ is a left uniserial ring, then every cyclic left $R$-module is uniserial, and hence uniform. Now assume that every cyclic left $R$-module is a direct sum of square-free modules. Since $R$ is a local ring, $R/I$ is indecomposable for every left ideal $I$ of $R$. Thus,  by hypothesis, $R/I$ and hence $J/I$ is square-free, for $J \leq R$ and so ${\rm Soc}(J/I)$ is simple or zero. Therefore, by \cite[Proposition 55.1]{Wisbauer}, $R$ is a left uniserial ring.
\end{proof}

%101010101010101010101010101010101010
%Examples%Examples%Examples%Examples
%101010101010101010101010101010101010
\section{\bf Examples}
In   this section, some relevant examples and counterexamples are included to illustrate
our results. First, we give an example of a ring $R$, which is both a right K{\"o}the ring and a  right co-K{\"o}the ring, but it is not a very stongly right co-K{\"o}the ring. 
\begin{Examp}\label{Asgar}{\rm (See also \cite[Example 1.22]{tuganbaev} and Nakayama \cite{Nakayama1, Nakayama2}).}
{\rm Let $R$ be the 5-dimensional algebra over the field $\mathbb{Z}_2$ generated by all $3 \times 3$-matrices of the form

$$\begin{bmatrix}
f_{11} & f_{12} & f_{13} \\ 0 & f_{22} & 0 \\ 0 & 0 & f_{33}
\end{bmatrix},$$
\noindent where $f_{ij} \in \mathbb{Z}_2$. Let $e_{ij}$ be the matrix whose $ij$th entry is equal to $1$ and all other entries are equal to $0$. Then $\{ e_{11}, e_{12}, e_{13}, e_{22}, e_{33} \}$ is a $\mathbb{Z}_2$-basis of the $\mathbb{Z}_2$-algebra $R$. Moreover,

\noindent  {{\rm (1)}} $1 =e_{11} + e_{22} + e_{33}$,  where $e_{11}, e_{22},e_{33}$ are local orthogonal idempotents,
$e_{12}\mathbb{Z}_2 =e_{12}R,$  $e_{13}\mathbb{Z}_2 = e_{13}R$, $J = e_{12}\mathbb{Z}_2 + e_{13}\mathbb{Z}_2$, $(J)^{2} = 0$, and the ring $R/J$ is isomorphic to a direct product of three copies of the field $\mathbb{Z}_2$.

\noindent {{\rm (2)}} $R_{R} = e_{11}R \oplus e_{22}R \oplus e_{33}R$, where $e_{22}R = e_{22}\mathbb{Z}_2$ and $e_{33}R = e_{33}\mathbb{Z}_2$ are simple projective right $R$-modules which are isomorphic to the modules $e_{12}R$ and $e_{13}R$, respectively.

\noindent {{\rm (3)}} $e_{11}R = e_{11}\mathbb{Z}_2+ e_{12}\mathbb{Z}_2+ e_{13}\mathbb{Z}_2$ is an indecomposable distributive (and hence square-free) Noetherian Artinian completely cyclic $R$-module,  but it is not uniform and every proper nonzero submodule of $e_{11}R$ coincides either with the projective module $e_{12}R \oplus e_{13}R$ or with one of the simple projective non-isomorphic modules $e_{12}R$ and $e_{13}R$.

 It can be checked that $R$ is a hereditary, right and left Artinian basic ring. So by \cite[Theorem 5.7]{Fazelpour2}, $R$ is a right K{\"o}the ring. It can be shown that $R$ has $32$ elements that $\vert U(R)\vert = 4$ and from $28$ nonunit elements of $R$ we have $3$ isomorphism classes of cyclic indecomposable modules, which are as follows:

$$Q_1 = e_{11}R, ~~~ Q_2 = \dfrac{e_{11}R}{e_{13}R}, ~~~ Q_3 = \dfrac{e_{11}R}{e_{12}R}, ~~~ Q_4 = \dfrac{e_{11}R}{J(R)}, ~~~ Q_5 = e_{22}R, ~~~ Q_6 = e_{33}R.$$

\noindent Also, every indecomposable cyclic right $R$-module has a square-free socle, since ${\rm Soc}(Q_1) = J(R)$ is square-free,  ${\rm Soc}(Q_2) = \dfrac{J(R)}{e_{13}R}$, ${\rm Soc}(Q_3) = \dfrac{J(R)}{e_{12}R}$, and $Q_4, Q_5$ and $Q_6$ are simple. Then every right  $R$-module is a direct sum of square-free modules ($R$ is a (strongly) right co-K{\"o}the ring) while the right $R$-module $e_{11}R$ is not a direct sum of uniform modules ($R$ is not a very strongly right co-K{\"o}the ring), since $R$ is not right serial.}
 \end{Examp}

We end the paper by showing that (very strongly) co-K{\"o}the is not right-left symmetric. The following result is helpful.

 \begin{Pro}\label{Cor3.6}
 Let $R$ be a local ring with the unique maximal ideal $\cal M$ such that ${_R\cal M}$ is simple, length$({\cal M}_R) = 2$ and length$(E(_R(R/{\cal M}))) = 3$. Then every left $R$-module is a direct sum of uniform modules.
 \end{Pro}
 \begin{proof}
Clearly, $R$ is an Artinian ring. Moreover, ${\cal M}^2 = 0$ since ${_R\cal M}$ is simple. So by \cite[Theorem 9]{Nicolson}, $R$ is  a left uniserial ring. Thus,  the left $R$-modules $R$,  $R/\mathcal{M}$ and $E(R/\mathcal{M})$ are uniform. Therefore, \cite[Theorem 3.6]{Behboodi2},  every left $R$-module is a direct sum of uniform modules.
 \end{proof}

\begin{Examp}\textup {(\cite[Example 3.7]{Behboodi2})}\label{exam}
{\rm Let $F$ be a field isomorphic to one of its proper subfield $\bar{F}$ such that $[F:\bar{F}] =2$ (for example let $F=Z_{2}(y)$, where $Z_{2}(y)$ is the quotient field of polynomial ring $Z_{2}[y]$ and let $\bar{F} = Z_{2}(y^{2})$, where $Z_{2}(y^{2})$ is the subfield of $F$ that has elements of the form $f(y^{2})/g(y^{2})$ such that $f(y^{2})$, $g(y^{2})$ are polynomials). Let $R=:F[x; \alpha]$ be the ring of polynomials of the form $a_{0}+a_{1}x$, $a_{i} \in F$ with multiplication defined by the rule $rx =x\alpha(r)$, where $\alpha$ is the isomorphism from $F$ to $\bar{F}$, and $x^{2} = 0$ together with the distributive law. Then by these relations $R$ is a local ring with maximal ideal $\mathcal{M} = xR$ (see \cite[p. 113, Bj{\"o}rk Example]{Jain1}). Let $\{1, a\}$ be a basis for the vector space $F$ over $\bar{F}$. It can be checked that $\mathcal{M} =xR=Rx \oplus Rxa$ with $\mathcal{M}^{2} = (0)$, and $(R \oplus R)/(x, xa)R$ is an injective right $R$-module. Also, $R$ is not a principal left ideal ring. Then by Proposition \ref{Cor3.6} and Theorem \ref{left co-Kothe}, every right $R$-module is a direct sum of uniform modules. Hence,  $R$ is a (very strongly) right co-K{\"o}the ring. Moreover, since $R$ is not   a left uniserial ring, by Corollary \ref{very strongly left co-Kothe}, $R$ is not a (very strongly) left co-K{\"o}the ring.}
\end{Examp}

\begin{Rem}
{\rm By Corollary \ref{Cor prime ideals commutes-Abelian}, in the local case, the condition ``every left and right module is a direct
sum of square-free modules" (very strongly co-K{\"o}the ring) is equivalent to ``every left and
right module is a direct sum of uniform modules", and this is equivalent to ``every right or left
module is a direct sum of distributive modules. By Theorem \ref{very strongly left co-Kothe}, every right module is a  direct sum  of square-free modules is equivalent to every right module is a direct sum of  uniform modules,  but it is not equivalent to every right module is a direct sum of distributive modules. Note that in Example \ref{exam}, $E((R/\mathcal{M})_{R})$ is a uniform module that is not distributive. Thus, uniform modules are square-free but are not necessarily distributive even though they are finitely generated and Artinian. Moreover, the ring $R$ in Example \ref{exam} shows that it is possible to have a
right (very strongly) co-K{\"o}the ring over which not every right module is semidistributive.}
\end{Rem}

 \end{document}